\documentclass[11pt,a4paper,reqno]{amsart}
\usepackage{amsmath}
\usepackage{amsfonts}
\usepackage{amssymb}
\usepackage{amscd}
\usepackage{url}
\usepackage{enumerate}
\usepackage[pdftex,bookmarks=true]{hyperref}
\usepackage{graphicx}

\newcommand\Vol{{\operatorname{Vol}}}

\newcommand\R{{\mathbf{R}}}

\renewcommand\P{{\mathbf{P}}}
\newcommand\E{{\mathbf{E}}}

\newcommand\M{{\operatorname{M}}}
\newcommand\dist{{\operatorname{dist}}}
\newcommand\Z{{\mathbf{Z}}}

\newcommand\col{{\mathbf{c}}}

\newcommand\ep{\varepsilon}
\newcommand\eps{\varepsilon}

\newcommand\Bc{{\mathbf c}}

\newcommand\Be{{\mathbf e}}

\newcommand\Bp{{\mathbf p}}

\newcommand\Bu{{\mathbf u}}
\newcommand\Bv{{\mathbf v}}

\newcommand\Bx{{\mathbf x}}
\newcommand\By{{\mathbf y}}
\newcommand\Bz{{\mathbf z}}

\newcommand\BB{{\mathbf B}}

\newcommand\BT{{\mathbf T}}

\newcommand\CE{{\mathcal E}}
\newcommand\CF{{\mathcal F}}
\newcommand\CG{{\mathcal G}}

\newcommand\CM{{\mathcal M}}
\newcommand\CN{{\mathcal N}}

\newcommand\CP{{\mathcal P}}

\newcommand\CS{{\mathcal S}}

\newcommand\LCD{\mathbf{LCD}}

\theoremstyle{plain}
  \newtheorem{theorem}[subsection]{Theorem}

  \newtheorem{fact}[subsection]{Fact}
  \newtheorem{lemma}[subsection]{Lemma}
  \newtheorem{corollary}[subsection]{Corollary}
  \newtheorem{question}[subsection]{Question}

  \newtheorem{remark}[subsection]{Remark}
  
  \newtheorem{claim}[subsection]{Claim}

\theoremstyle{definition}
  \newtheorem{definition}[subsection]{Definition}

\author{Hoi H. Nguyen}
\address{Department of Mathematics, The Ohio State University, Columbus OH 43210}

\email{nguyen.1261@math.osu.edu}

\thanks{The author is supported by grants DMS-1600782, DMS-1128155,  and CCF-1412958. The most part of this note was done while the author was visiting IAS and VIASM. The author thanks the two institutions for their hospitality.}


\begin{document}

\title[Asymptotic Lyapunov exponents for large random matrices]{Asymptotic Lyapunov exponents for large random matrices}

\begin{abstract} Suppose that $A_1,\dots, A_N$ are independent random matrices whose atoms are iid copies of a random variable $\xi$ of mean zero and variance one. It is known from the works of Newman et. al. in the late 80s that when $\xi$ is gaussian then $N^{-1} \log \|A_N \dots A_1\|$ converges to a non-random limit. We extend this result to more general matrices with explicit rate of convergence. Our method relies on a simple connection between structures and dynamics.
\end{abstract}

\maketitle
\section{Introduction}
Let $A_i, i\ge 1$ be a sequence of independent identically distributed square random matrices of a given distribution $\mu$ in $\M_\R(n)$. Let $B_N$ be the matrix product 

$$B_N= A_N \dots A_1.$$

Furstenberg and Kesten \cite{FK}  (see also \cite[Theorem 4.1, p.11]{BL}) proved in 1960 that

\begin{theorem} Assume that $\E \log^+ (\|A_i\|)<\infty$ (where $\log^+ x =\max\{0,\log x\}$) then with probability one $\frac{1}{N} \log \|B_N\|$ converges to a deterministic number $\gamma$. 
\end{theorem}

Here and later, if not specified, our norm is always the $\|.\|_2$ norm. The limit $\gamma$ is called the {\it top Lyapunov exponent}. If we assume the common distribution $\mu$ of the $A_i$ to be strongly irreducible (i.e. there does not exist a finite family of proper linear subspaces $V_1,\dots,V_k$ of $\R^n$ such that $M_\mu(V_1\cup \dots\cup V_k )=V_1 \cup \dots \cup V_k$, where $M_\mu$ is the smallest closed subgroup which contains the support of $\mu$), then Furstenberg showed in \cite{F} (see also \cite[Corollary 3.4, p.53]{BL}) that 

\begin{theorem}[Furstenberg's theorem]\label{theorem:F} Assume that $\E \log^+ (\|A_i\|)<\infty$ and that $\mu$ is strongly irreducible, then
\begin{itemize}
\item $\lim_{N\to \infty} \frac{1}{N} \log \|B_N \Bx\| = \gamma$ uniformly on $\Bx \in S^{n-1}$;
\vskip .1in
\item for any $\mu$-invariant distribution $\nu$ on $\P(\R^n)$ (i.e. $\nu(A) = \int \int 1_A(M \bar{\Bx}) d\mu(M) d\nu (\bar{\Bx})$) we have
$$\gamma = \int \int \log \frac{\|M\Bx\|}{\|B_N \Bx\|} d \mu(M) d\nu(\bar{\Bx}).$$
\end{itemize}
\end{theorem}

There are also important extensions when $M_\mu$ is replaced by $T_\mu$, the smallest closed semi-group which contains the support of $\mu$; and when strongly irreducibility is reduced to irreducibility, see for instance \cite{BL,F,GM}.


We next introduce other Lyapunov exponents by the use of exterior powers $\wedge^k$. 

\begin{definition} Assume that  $\E \log^+ (\|A_i\|)<\infty$. The Lyapunov exponents $\gamma_1,\dots, \gamma_n$ associated to $A_i$ are defined inductively by $\gamma_1=\gamma$ and for $k\ge 2$,

$$\sum_{i=1}^k \gamma_i = \lim_{N \to \infty} \frac{1}{N} \E \log \|\wedge^k B_N\|.$$

\end{definition}

In \cite{O} (see also \cite[Theorem 1.2]{GM}), Oseledec showed the following extremely powerful theorem on the convergence of Lyapunov exponents.  

\begin{theorem}[Oseledec's multiple ergodic theorem]\label{theorem:O} Assume that  $\E \log^+ (\|A_i\|)<\infty$, then the followings hold.
\begin{itemize} 
\item With probability one,

\begin{equation}\label{eqn:def:L}
\gamma_k = \lim_{N \to \infty} \frac{1}{N} \E \log  \sigma_k(B_N),
\end{equation}

where $\sigma_1(B_N)\ge \dots \ge \sigma_n(B_N)$ are the singular values of $B_N$.
\vskip .1in

\item With probability one, the matrix limit $(B_N B_N^T)^{1/2N}$ converges to a matrix $M \in M_\R(n)$ whose eigenvalues coincide with $\exp(\gamma_i)$ counting  multiplicities.
\vskip .1in
\item Let $\exp(\alpha_1(M))<\dots < \exp(\alpha_k(M))$ denote the different eigenvalues of $M$ with multiplicities $n_1(M),\dots, n_k(M)$, and let $U_1,\dots, U_k$ be the corresponding eigensubsapces, and set $V_i = U_1 \oplus \dots \oplus U_i$. Then the pair $(\alpha_i(M), n_i(M))$ is $\mu$-invariant, and for any unit vector $\Bx\in V_i \backslash V_{i-1}$, with probability one

$$\lim_{N\to \infty} \frac{1}{N} \log \|B_N \Bx\| = \alpha_i.$$
\end{itemize}
\end{theorem}

In practice, the issues when the top exponent $\gamma_1$ is strictly positive or when all of the Lyapunov exponents are distinct are extremely important.  We refer the reader to \cite{GM} for further discussion on these accounts.

Following the two celebrated results of Furstenberg and Oseledec above, for some {\it nice} distribution $\mu$ it is also natural to ask

\begin{question}\label{question} Can we give
\begin{enumerate}[(i)]
\item  (invariant measure) fine approximation for the Lyapunov's exponents? 
\vskip .05in
\item (large deviation type) quantification of the rate of convergence ?
\end{enumerate}
\end{question}

These aspects have been widely studied by many researchers, especially for unimodular and/or symplectic matrices of fixed size in connection to the theory of Schr\"odinger operators. For a thorough introduction to these topics, we refer the reader to the books by Figotin and Pastur \cite{FP} and by Bourgain \cite{B}. For the sake of completeness, allow us to insert here a large deviation type result for the shift model from \cite{B} (see also \cite{BS} and \cite{GS}).

\begin{theorem} Assume that $\omega$ is an element of the one dimensional torus $\BT$ such that 

$$\dist(k \omega, \Z^2) > c \frac{1}{|k| \log^3(1+|k|)} \mbox{ for all } k \in \Z/\{0\}.$$
Let $E$ be a fixed parameter and let $f$ be a real analytic function on $\BT$. Let $\Bx$ be sampled uniformly at random from $\BT$, and consider the random matrix product 

$$B_N = \prod_{j=1}^N \begin{pmatrix} f(\Bx+ j \omega)-E & -1 \\ 1 & 0 \end{pmatrix}.$$ 

Then for $t > N^{-1/10}$

$$\P_{\Bx} \Big(\big|\frac{1}{N} \log \|B_N\| - \frac{1}{N}\E \log \|B_N\|\big| > t\Big) < C e^{-ct^2N},$$

for some absolute constants $C$ and $c$.

\end{theorem}

One can use result of this type to study the decay of the corresponding Green function.

\subsection{The iid model with large dimension} Our main focus is on a simple model of random matrices of very {\it large} dimension which are not necessarily unimodular. Especially, we will consider those $A_i$ random matrices where the entries are iid copies of a common real random variable $\xi$ of mean zero and variance $1/n$. This ensemble had been considered by Cohen, Isopi and Newman in the  80s \cite{CN,IN,N} in connection to May's proposal of a specific quantitative relationship between complexity and stability within certain ecological models. We cite here a result by Newman which might best suit our discussion.

\begin{theorem}\cite[Equation (6)]{N}\label{theorem:IN} Assume that the entries of $A_i$ are iid copies of $\frac{1}{\sqrt{n}}N(0,1)$. Let $\mu_1\ge \dots \ge \mu_n$ be the Lyapunov's exponents of the matrix product $B_N$. Then 
$$\mu_i = \frac{1}{2} (\log 2 + \Psi(\frac{n-i+1}{2}) -\log n),$$

where $\Psi(d)=\Gamma'(d)/\Gamma(d)$ is the digamma function.
\end{theorem}


This result was also generalized in \cite{IN} to $\xi$ having bounded density and $\E ((\sqrt{n} \xi)^4)<\infty$. We also refer the reader to a more recent result by Forester \cite{Forester} and the survey \cite{AI} by Akerman and Ispen for more references. These results address the first part of Question \ref{question} for various random matrices of {\it invariance} type. 

For the large deviation part of Question \ref{question}, the only result we found for the iid model is due to Kargin \cite{K} who considered the rate of convergence of the top exponents. 

\begin{theorem}\cite[Proposition 3]{K}\label{theorem:K}  Let $\eps>0$ be given. Assume that the entries of $A_i$ are iid copies of $\frac{1}{\sqrt{n}}N(0,1)$. Then for all sufficiently small $t$, and all $n\ge n_0(t)$ and $N\ge 1$

$$\P(|\frac{1}{N} \log \|B_N\|| >t + \eps/N) \le 2 (1+2/\eps)^n \exp(-\frac{1}{8} Nn t^2).$$
\end{theorem}

\begin{remark}
To be more precise, Proposition 3 of \cite{K} shows that  $\P(|\frac{1}{N} \log \|B_N\Bx\|| >t ) \le \exp(-\frac{1}{8} Nn t^2)$ for any fixed $\Bx\in S^{n-1}$, from which one can deduce Theorem \ref{theorem:K} by an $\eps$-net argument, see for instance Claim \ref{claim:netpassing}. 
\end{remark}

\subsection{Our results} As far as we are concerned, all of the results in the literature with respect to this iid model assumed the common distribution $\xi$ to be sufficiently  smooth (i.e. at least the density function exists and is bounded), so that  $\frac{1}{N}\log \|B_N\|$ with $N\to \infty$ is well defined almost surely. 

The smoothness assumption is natural, as if $A_i$ were singular with positive probability, then our chain $B_N$ would become singular with probability one; in this case it might still be reasonable to study the top Lyapunov exponent, but not other exponents. However, and this is important to our study, {\it would it still be useful to study the "Lyapunov exponents" when $N$ does not necessarily tend to infinity?} In other words, even when the exponents are not well defined, can we still say useful things about the growth of the chain $B_N$ for some {\it effective} range of $N$? Our motivation is  based on the following two facts:

\begin{enumerate}[(i)]
\item in many practical problems, it is not known a priori that our random matrix model is smooth;
\vskip .05in
\item to estimate the Lyapunov's exponents using computer, it  actually computes $\frac{1}{N}\log \sigma_i(B_N)$ for some sufficiently large (but not too large) $N$.
\end{enumerate}

Trying to address these questions, and with a universality approach in mind, we will consider the matrix models  $A_i$ where the entries of $\sqrt{n} A_i$ are iid copies of a random variable $\xi$ of mean zero, variance one, and that there exists a parameter $K$ such that for all $t$

\begin{equation}\label{eqn:subgaussian}
\P(|\xi| \ge t)= O(\exp(-t^2/K)).
\end{equation}

One representative example of our matrices are Bernoulli ensembles, where $\xi$ takes value $\pm 1$ with probability 1/2. As addressed above, there are two main obstacles for this discrete model: first the matrix law is not rotational invariant; and second, with probability one, the product matrix $B_N$ will be the {\it zero matrix} as $N \to \infty$ (for instance when $n$ is even then with positive probability the tries of $A_1$ are all 1s and each row of $A_2$ has exactly $n/2$ entries 1s).

The first problem is not strictly impossible, as there have been major developments in recent years showing that the spectral behavior of the iid matrices is universal. The second problem is, on the other hand, more subtle. This forces us to put an upper bound on $N$. The main question then boils down to finding a fine range of $N$ for which one can still achieve non-trivial estimates.

In this note we show that as long as $N$ grows slower than exponential in $n$, one can have relatively good control on the exponents. 

\begin{theorem}[Main results]\label{theorem:main} Let $\ep>0$ be given. Let $A_1,A_2,\dots$ be independent matrices whose entries are iid copies of $\frac{1}{\sqrt{n}}\xi$ with  $\xi$ from \eqref{eqn:subgaussian}. Then the followings hold.

\begin{enumerate}[(1)] 
\item (top exponent) For any $ t\ge 1/n$ we have 

$$\P\Big( |\frac{1}{N}  \log \|B_N\| | \ge t + \eps/N) \Big) \le (1+2/\eps)^n \big[\exp(-c\min\{t^2, t \} Nn) + N n^{-cn}\big].$$

\item (second exponent) For any $ t\ge 1/n$ we have 

$$\P\Big( |\frac{1}{N}  \log \sup_{(\Bx_1,\Bx_2)\in \Gamma_2} \|B_N\Bx_1\wedge B_N\Bx_2\|  | \ge t + \eps/N) \Big) \le (1+2/\eps)^n \big[\exp(-c\min\{t^2, t \} Nn) + N n^{-cn}\big].$$

\item (last exponent) For any constant $\eps>0$, there exists $C=C(\eps)$ such that

$$\P\Big( \inf_{\Bx\in S^{n-1}} \frac{1}{N} \log \|B_N \Bx\| \le -(\frac{1}{2}+\eps) \log n \Big) \le C^n \exp(-N/2) + Nn^{-\omega(1)}.$$
\end{enumerate}
\end{theorem}

In short, (1) of Theorem \ref{theorem:main} extends Theorem \ref{theorem:K} to general matrix ensemble, with the extra assumption $N \ll n^{cn}$ and with $n_0(t)=O(1/t)$. It shows that although the chain dies out eventually, one (and the computer) can still see the concentration of the top exponent as long as $N$ is not exceedingly large. This also fits with the simulation presented in \cite{K1} (see Fig. 4 and the discussion afterwards.) By taking $t = 1/n$ and $\eps =1/2$, we obtain

$$\P\Big( |\frac{1}{N}  \log \|B_N\| | \ge O(1/n) \Big) \le   C^n[\exp(-N/n) +N n^{-cn}] .$$

We also show that the approach can be modified (in a non-trivial way)  to control other top Lyapunov's exponents: it follows from (1) and (2) that the asymptotic second exponent $\gamma_2$ is also well concentrated around zero, and the method seems to extend to other asymptotic $\gamma_k$ for any fixed $k$. Nevertheless, our concentration result is not local enough to see the difference between $\gamma_1$ and $\gamma_2$ as in Theorem \ref{theorem:IN}.

Finally, we  show in (3) that the asymptotic least exponent $\gamma_n$ is approximately at least $-\frac{1}{2}\log n$, which again fits with the calculation of Theorem \ref{theorem:IN}. Our control for $\gamma_n$, on the other hand, is not as sharp as for the top ones. It is not clear how $\gamma_n$ fluctuates around its mean, but we believe that it is not very well concentrated. Furthermore, a similar bound for $\P(\inf_{\Bx\in S^{n-1}} \frac{1}{N} \log \|B_N \Bx\| \ge -(\frac{1}{2}-\eps) \log n)$ is expected to hold, but we will not address this matter here (it is usual the case that bounding this quantity from below (i.e. (3)) is essentially harder than from bounding from above.)

In our next section we introduce the methods to prove Theorem \ref{theorem:main}.

\section{Proof method}\label{section:method}

\subsection{The top exponent}\label{subection:(1)} Here we discuss the method to prove (1) of Theorem \ref{theorem:main}. To estimate $\|B_N\| = \sup_{\Bx_0 \in S^{n-1}} \|B_N\Bx_0\|$, it is worth working with a finite collection of unit vectors $\Bx_0$. Let $\eps>0$ be a parameter, and let $\CN_{start}$ be an $\eps$-net of $S^{n-1}$. It is well-known that one can assume $|\CN_{start}| \le (1+2/\eps)^n$. The following is often used in the context of bounding the largest singular values of random matrices.

\begin{claim}\label{claim:netpassing} We have
$$\sup_{\Bx_0\in \CN_{start}} \|B_N \Bx_0\| \le \|B_N\|  \le (1- \eps)^{-1} \sup_{\Bx_0\in \CN_{start}} \|B_N \Bx_0\|.$$
\end{claim}

With this claim, one hopes to control $\frac{1}{N}\log \|B_N\|$ (up to an approximated factor $1+\frac{1}{N}\log (1+\eps)$ and up to a correcting factor $(1+2/\eps)^n$ in probability) by establishing a strong concentration result for $\frac{1}{N} \log \|B_N \Bx_0\|$ for each $\Bx_0\in \CN_{start}$. This was also the main starting point of \cite{K}.

Let $\Bx_0$ be an element of $\CN_{start}$. One might write

\begin{align*}
\log \|B_N\Bx_0 \|&= \log \|A_N A_{N-1} \dots A_2 A_1 \Bx_0\|\\
& = \log \|A_N \frac{A_{N-1} \dots A_2 A_1 \Bx_0}{\|A_{N-1} \dots A_2 A_1 \Bx_0\|}\|+ \log \|A_{N-1} \frac{A_{N-2} \dots A_2 A_1 \Bx_0}{\|A_{N-2} \dots A_2 A_1 \Bx_0\|}\|+\dots+ \\
&+ \dots + \log \| A_2 \frac{A_1 \Bx_0}{\|A_1 \Bx_0\|}\| + \log \|A_1\Bx_0\|\\
&= \sum_{i=0}^{N-1} \log \|A_{i+1}\Bx_i \|,
\end{align*}

where 

\begin{equation}\label{eqn:x_i}
\Bx_i: =  \frac{A_{i} \dots A_2 A_1 \Bx_0}{\|A_{i} \dots A_2 A_1 \Bx_0\|}.
\end{equation}

When $\xi$ has discrete distribution such as Bernoulli, there would be a minor problem that $B_N \Bx_0$ might be vanishing, but we could rule out this possibility by choosing the net $\CN_{start}$ to consist of vectors of "highly irrational" entries (which remain highly irrational under the actions of the matrices $A_i$). 

Now we want to control $\log \| A_{i+1} \Bx_i\|$ conditioning on $A_1,\dots, A_i$ (and hence on $\Bx_i$). Note that 

$$\E_{A_{i+1}} \|A_{i+1} \Bx_i\|^2 = 1.$$ 

Roughly speaking, to hope for a good concentration of $\log \| A_{i+1} \Bx_i\|$ around zero, the very first step we have to guarantee is that with high probability with respect to $A_{i+1}$, the vector norm $\| A_{i+1} \Bx_i\|$ is being well away from zero. 

This probability certainly depends on the structure of $\Bx_i$. For instance if $\Bx_i =(\pm 1/\sqrt{2}, \pm 1/\sqrt{2},0,..,0)$ or $\Bx_i=(\pm 1\/\sqrt{n}, \dots, \pm 1/\sqrt{n})$ then the chance that $\|A_{i+1}\Bx_i\|$ being small (or even being annihilated) is not quite small if we are working with Bernoulli matrices. With this in mind, our general strategy consists of three main steps.

\begin{itemize}
\item (Step 1.) (dynamics and structures) find a set $\CS$ of $S^{n-1}$ with the following properties:

\begin{itemize} 
\item $\CS$ covers an $\eps$-net $\CN_{start}$ of $S^{n-1}$,
\item $\CS$ remains stable under the action of each $A_i$. In other words, with high very high probability all of the normalized vectors $\Bx_i$ from \eqref{eqn:x_i} belong to $\CS$;
\item  for any $\Bx\in \CS$, with high probability with respect to $A_{i+1}$ the norm $\|A_{i+1} \Bx\|$ is bounded away from zero.
\end{itemize}
\vskip .1in
\item (Step 2.) (concentration over good vectors) we show that for each $\Bx_i \in \CS$,  $\log \| A_{i+1} \Bx_i\|$ is very well concentrated around zero.
\vskip .1in
\item (Step 3.) (law of large number) use concentration information from Step 2 to prove Theorem \ref{theorem:main} .
\end{itemize}

We will lay out the choice of $\CS$ in Section \ref{section:step1}. Step 2 will be carried out in Section \ref{section:step2}, and Step 3 is concluded in Section \ref{section:step3}.

\subsection{The second exponent}\label{subsection:(2)} We will extend the ideas of the previous subsection to deal with (2) of  Theorem \ref{theorem:main}. First of all, let $\CP_{start}$ be some subset of $S^{n-1} \times S^{n-1}$ that covers an $\eps$-net (that is  for any $(\Bx, \By) \in S^{n-1} \times S^{n-1}$ there exists $(\Bx',\By') \in \CP_{start}$ such that $\|\Bx-\Bx'\|, \|\By-\By'\|\le \eps$). Similarly to Claim \ref{claim:netpassing}, we have the following.

\begin{claim}\label{claim:netpassing:2}
$$ \sup_{(\Bx,\By)\in S^{n-1} \times S^{n-1}} \Vol_2 (B_N\Bx, B_N\By) \le (1 -2\eps -\eps^2)^{-1}  \sup_{(\Bx',\By')\in  S^{n-1} \times S^{n-1}\cap \CP_{start} } \Vol_2 (B_N\Bx', B_N\By') .$$
\end{claim}

\begin{proof}(of Claim \ref{claim:netpassing:2})
Assume that $\sup_{(\Bx,\By)\in  S^{n-1} \times S^{n-1} } \Vol_2 (B_N\Bx, B_N\By)$ is attained at $(\Bx,\By)$. Let $(\Bx',\By')$ be an element in $\CP_{start}$ such that $\|\Bx-\Bx'\|\le \eps$ and $\|\By-\By'\|\le \eps$. By the triangle inequality 
\begin{align*}
\Vol_2 (B_N\Bx, B_N\By) &\le \Vol_2 (B_N\Bx', B_N\By') + \Vol_2 (B_N (\Bx-\Bx'),B_N\By)\\
&+ \Vol_2 (B_N \Bx',B_N(\By-\By'))+ \Vol_2(B_N (\Bx-\Bx'), B_N (\By-\By'))\\
&\le \Vol_2 (B_N\Bx',B_N\By')  + (2\eps+\eps^2) \sup_{(\Bz_1,\Bz_2)\in  S^{n-1} \times S^{n-1} } \Vol_2 (B_N\Bz_1, B_N\Bz_2).
\end{align*}
\end{proof}

Beside containing an $\eps$-net we will also choose  $\CP_{start} \subset S^{n-1} \times S^{n-1}$ to satisfy certain non-structured properties (such as $\CP_{start} \subset \CP$, a broader set to be introduced below). The detail of construction of $\CP_{start}$ will be presented in Section \ref{section:second}.

Now let $(\Bx_0,\By_0) \in \CP_{start}$. As customary, one might write

\begin{align*}
\log \|B_N\Bx_0 \wedge B_N \By_0 \|&= \log \|A_N A_{N-1} \dots A_2 A_1 \Bx_0 \wedge A_N A_{N-1} \dots A_2 A_1 \By_0 \|\\
& = \log \|A_N \frac{A_{N-1} \dots A_2 A_1 \Bx_0 \wedge A_{N-1} \dots A_2 A_1 \By_0 }{\|A_{N-1} \dots A_2 A_1 \Bx_0 \wedge A_{N-1} \dots A_2 A_1 \By_0 \|} \|\\ 
&+ \log \|A_{N-1} \frac{A_{N-2} \dots A_2 A_1 \Bx_0 \wedge A_{N-2} \dots A_2 A_1 \By_0 }{\|A_{N-2} \dots A_2 A_1 \Bx_0 \wedge A_{N-2} \dots A_2 A_1 \By_0 \|}\|+ \dots\\ 
&+ \log \| A_2 \frac{A_1 \Bx_0 \wedge A_1 \By_0}{\|A_1 \Bx_0 \wedge A_1 \By_0 \|}\| + \log \|A_1\Bx_0 \wedge A_1 \By_0\|\\
&= \sum_{i=0}^{N-1} \log \frac{\|A_{i+1}\Bx_i \wedge A_{i+1}\By_i\|}{\|\Bx_i \wedge \By_i\|},
\end{align*}

where 

$$\Bx_i: =  \frac{A_{i} \dots A_2 A_1 \Bx_0}{\|A_{i} \dots A_2 A_1 \Bx_0\|} \mbox{ and } \By_i: =  \frac{A_{i} \dots A_2 A_1 \By_0}{\|A_{i} \dots A_2 A_1 \By_0\|}.$$

To control $\frac{\|A_{i+1}\Bx_i \wedge A_{i+1}\By_i\|}{\|\Bx_i \wedge \By_i\|}$, we first pull out $\|A_{i+1}\Bx_i\|$ and $\|A_{i+1}\By_i\|$

\begin{align*} 
\log \frac{\|A_{i+1}\Bx_i \wedge A_{i+1}\By_i\|}{\|\Bx_i \wedge \By_i\|} &= \log \|A_{i+1} \Bx_i\| + \log  \|A_{i+1} \By_i\| \\
&+\log \frac{\|A_{i+1}\Bx_i/\|A_{i+1}\Bx_i\| \wedge A_{i+1}\By_i/\|A_{i+1}\By_i\|\|}{\|\Bx_i \wedge \By_i\|}\\
&= \log \|A_{i+1} \Bx_i\| + \log  \|A_{i+1} \By_i\| + \log \frac{\|\Bx_{i+1} \wedge \By_{i+1}\|}{\|\Bx_i \wedge \By_i\|}
\end{align*}

By our treatment of the top exponent, one has very good control on $ \log \|A_{i+1} \Bx_i\| + \log  \|A_{i+1} \By_i\|$, thus the main task is to study the remaining  term. To hope for a good concentration of $\log \frac{\|\Bx_{i+1} \wedge \By_{i+1}\|}{\|\Bx_i \wedge \By_i\|}$ around zero, among other things we have to guarantee that $\Bx_i \wedge \By_i \neq 0$ with high probability, and within this event that $ \frac{\|\Bx_{i+1} \wedge \By_{i+1}\|}{\|\Bx_i \wedge \By_i\|}$ is close to one. Thus compared to the previous section, beside the task of bounding $\|\Bx_{i+1}\|, \|\By_{i+1}\|, \|\Bx_i\|, \|\By_i\|$ away from zero, we will have to show that the angles between these vectors are highly stable under the process, and this task is more complicated. Nevertheless, our plan remains the same in principle.

\begin{itemize}
\item (Step 1.) (dynamics and structures) find a set $\CP$ of pair vectors in $\R^n$ such that $\CP_{start} \subset \CP$ and which remains stable under the action of the $A_i$'s with given $(\Bx_i,\By_i)\in \CP$: with very high probability with respect to $A_{i+1}$ 

$$(\Bx_{i+1}, \By_{i+1}) =(\frac{A_{i+1}\Bx_i}{\|A_{i+1} \Bx_i\|}, \frac{A_{i+1}\By_i}{\|A_{i+1} \By_i\|}) \in \CP.$$
\vskip .1in
\item (Step 2.) (concentration over good vectors) show that for given $(\Bx_i,\By_i) \in \CP$, with very high probability with respect to $A_{i+1}$ the norm $\|\frac{\Bx_{i+1} \wedge \By_{i+1}}{\|\Bx_i \wedge \By_i\|}\|$ is very well concentrated around one.
\vskip .1in
\item (Step 3.) (law of large number) use concentration information from Step 2 to prove (2) of Theorem \ref{theorem:main}.
\end{itemize}

We will present a full proof of (2) of Theorem \ref{theorem:main} with a more detailed description of the three steps above in Section \ref{section:second}.

\subsection{The last exponent}\label{subsection:(3)} Now we discuss the method to prove (3) of Theorem \ref{theorem:main}. Here the net argument does not work at all. We will have to relate the smallest Lyapunov exponent to the distances among the rows of the matrices $A_i$.

Let $\eps>0$ be a given small constant, and consider the event 

$$\CE_\eps=\Big\{\inf_{\Bx\in S^{n-1}} \|B_N \Bx\| \le T_\eps:= ((1-\eps)/\sqrt{n})^N\Big \} .$$
 
As $\Bx=(x_1,\dots,x_n) \in S^{n-1}$, there exists $i_0\in [n]$ such that $|x_{i_0}|\ge 1/\sqrt{n}$. With $\Bc_i = B_N \Be_i$ being the $i$-th column vector of $B_N$, it follows from $\|\sum_{i} x_i \col_i\|\le T_\eps$ that 

$$\dist(\col_{i_0}, span(\col_i, i \neq i_0)) \le \sqrt{n} T_\eps.$$

Let $\CE_{\eps,1}$ be the event that 

$$\CE_{\eps,1} = \Big\{\log \dist(\col_{n}, span(\col_i, i \neq n)) \le \log \sqrt{n} + \log  T_\eps \Big \}.$$ 

We then have 

$$\P(\CE_\eps) \le n\P(\CE_{\eps,1}).$$

Thus for the upper bound, the main focus is to estimate $\P(\CE_{\eps,1})$. We will show that this probability is so small that the extra factor $n$ will not affect at all.

In general, for any general non-degenerate tuple $(\Bv_1,\dots,\Bv_n)$, 

$$\|\frac{A\Bv_1 \wedge \dots \wedge A \Bv_n}{\Bv_1 \wedge \dots \wedge \Bv_n}\| = \frac{\sqrt{\det (V^\ast A^\ast A V)}}{\sqrt{V^\ast V}}= |\det(A)|.$$ 

Also, for any non-degenerate tuple $(\Bv_1,\dots, \Bv_{n-1})$, with $\Bv_n \in S^{n-1}$ being orthogonal to all other $\Bv_i, 1\le i\le n-1$, we write

\begin{align*}
\|\frac{A\Bv_1 \wedge \dots \wedge A \Bv_{n-1}}{\Bv_1 \wedge \dots \wedge \Bv_{n-1}}\| &= \frac{|\det(A\Bv_1, \dots, A \Bv_{n-1})|}{|\det(\Bv_1, \dots, \Bv_{n-1})|}\\
& = \frac{|\det(A\Bv_1, \dots, A \Bv_{n-1}, A\Bv_n)|/ \dist(A\Bv_n, H_{A\Bv_1,\dots,A\Bv_{n-1}}) }{|\det(\Bv_1, \dots, \Bv_{n-1},\Bv_n)|}\\
&=  \frac{|\det(A)|}{\dist(A\Bv_n, H_{A\Bv_1,\dots,A\Bv_{n-1}}) },
\end{align*}

where $H_{A\Bv_1,\dots,A\Bv_{n-1}}$ is the subspace spanned by $A\Bv_1,\dots, A\Bv_{n-1}$. Taking $\Bv_i$ to be the standard normal basis $\Be_i$ we thus obtain,

\begin{align*}
\log \dist(\col_{n}, span(\col_i, i \neq n))  & = \log \dist(B_N\Be_n, H_{B_N\Be_1,\dots,B_N\Be_{n-1}}) \\
&= \log \|\frac{B_N\Be_1 \wedge \dots \wedge B_N \Be_n}{\Be_1 \wedge \dots \wedge \Be_n}\| - \log \|\frac{B_N\Be_1 \wedge \dots \wedge B_N \Be_{n-1}}{\Be_1 \wedge \dots \wedge \Be_{n-1}}\|.
\end{align*}

Now as $B_N = A_N \dots A_1$, we can rewrite the second term of the above formula as 

\begin{align*}
&\log \|\frac{B_N\Be_1 \wedge \dots \wedge B_N \Be_{n-1}}{\Be_1 \wedge \dots \wedge \Be_{n-1}}\|  = \log \|\frac{A_N \dots A_1\Be_1 \wedge \dots \wedge A_N \dots A_1 \Be_{n-1}}{\Be_1 \wedge \dots \wedge \Be_{n-1}}\|\\
& \log \|\frac{A_N (A_{N-1} \dots A_1\Be_1) \wedge \dots \wedge A_N(A_{N-1} \dots A_1 \Be_{n-1})}{A_{N-1} \dots A_1\Be_1 \wedge \dots \wedge A_{N-1} \dots A_1 \Be_{n-1})}\| + \log \|\frac{A_{N-1} \dots A_1\Be_1 \wedge \dots \wedge A_{N-1} \dots A_1 \Be_{n-1}}{A_{N-2} \dots A_1\Be_1 \wedge \dots \wedge A_{N-2} \dots A_1\Be_{n-1}}\| \\
&+ \dots + \log \|\frac{A_1\Be_1 \wedge \dots \wedge A_1 \Be_{n-1}}{\Be_1 \wedge \dots \wedge \Be_{n-1}}\|.
\end{align*}

Decomposing similarly for $\log \|B_N\Be_1 \wedge \dots \wedge B_N \Be_{n}/\Be_1 \wedge \dots \wedge \Be_{n}\|$, we obtain

\begin{equation}\label{eqn:decomposition:dist}
\log \dist(B_N\Be_n, H_{B_N\Be_1,\dots,B_N\Be_{n-1}}) = \sum_i \log \dist(A_i \Bv_i, H_{A_i \dots A_1 \Be_1,\dots,A_i \dots A_1 \Be_{n-1}}),
\end{equation}

where $\Bv_i$ is a unit vector that is orthogonal to the vectors $\Bu_{i1},\dots, \Bu_{i (n-1)}$ satisfying

$$\Bu_{i1}=A_{i-1}\dots A_1 \Be_1,\dots, \Bu_{i (n-1)}=A_{i-1}\dots A_1 \Be_{n-1}.$$ 

Now as $(A_i^{-1}\Bv_i)^{T} A_i \Bu_{ij} = \Bv_i^T\Bu_{ij}=0$, the vector $A_i^{-1}\Bv_i/\|A_i^{-1}\Bv_i\|$ is the unit normal vector of the subspace $H_{A_i\Bu_{i1},\dots, A_i\Bu_{i(n-1)}}$, and so

\begin{equation}\label{eqn:formula:dist} 
\dist(A_i \Bv_i, H_{A_i \dots A_1 \Be_1,\dots,A_i \dots A_1 \Be_{n-1}}) = \frac{1}{\|A_i^{-1}\Bv_i\|}.
\end{equation}

Note that the vectors $\Bu_{i1},\dots, \Bu_{i(n-1)}$ and $\Bv_i$ are independent of $A_i$, and hence it boils down to study the upper bound of $\|A_i^{-1}\Bv_i\|$ with the randomness with respect to $A_i$.  We will study this in more detail in Section \ref{section:least}.

The rest of the paper is organized as follows. We will provide a detailed treatment for (1) of Theorem \ref{theorem:main} throughout sections \ref{section:step1}, \ref{section:step2} and \ref{section:step3}. We then sketch the proof of (2) of Theorem \ref{theorem:main} in Section \ref{section:second}. The proof of (3) of Theorem \ref{theorem:main} will be presented in Section \ref{section:least}. We conclude the note by Section \ref{section:remark} with some remarks. 

{\bf Notation}. Throughout this paper, we regard $N$ as an asymptotic parameter going to infinity. We write $X = O(Y ), X\ll Y$, or $Y\gg X$ to denote the claim that $|X| \le CY$ for some fixed $C$; this fixed quantity $C$ is allowed to depend on other fixed quantities such as the sub-gaussian parameter $K$ of $\xi$ unless explicitly declared otherwise. We also use $o(Y )$ to denote any quantity bounded in magnitude by $c(N)Y$ for some $c(N)$ that goes to zero as $N\to \infty$.

\section{Step 1 for (1) of Theorem \ref{theorem:main}: structures under matrix action}\label{section:step1}  Our choice of the set $\CS$ is  motivated by recent ideas from Tao-Vu \cite{TV1,TV2} and from Rudelson-Vershynin \cite{rv,rv-rec} in the context of controlling the small ball probability of random walk. Although this looks surprising at first, the reader will see later that these structures are indeed the right  object to work with.

We first introduce the notion of {\it least common denominator} by Rudelson and Versynin (see \cite{rv}). Fix parameters $\kappa$  and $\gamma$, where $\gamma \in (0,1)$. For any nonzero vector $\Bx$ define
$$
\LCD_{\kappa,\gamma}(\Bx)
:= \inf \Big\{ \theta> 0: \dist(\theta \Bx, \Z^n) < \min (\gamma \| \theta \Bx \|,\kappa) \Big\}.
$$

We record a few easy consequences of $\LCD$.

\begin{fact}\label{fact:approxLCD} We have

\begin{itemize} 
\item If $\By = \lambda \Bx$ with $\lambda \neq 0$, then 

$$\LCD_{\kappa, \gamma}(\By) = \frac{1}{|\lambda|} \LCD_{\kappa, \gamma}(\Bx).$$
\item Assume that $\|\Bx\|, \|\By\| \ge \eps$ with $D=\LCD_{\gamma,\kappa}(\Bx) \ge 1$ and $\|\Bx-\By\|  \le D^{-2}$, then 

$$\LCD_{\kappa+1,\gamma+\frac{1}{D}}(\By) \le \LCD_{\kappa,\gamma}(\Bx).$$
\end{itemize}
\end{fact}

\begin{proof}(of Fact \ref{fact:approxLCD}) Assume that $\dist(D \Bx, \Z^n) \le \min \{ \gamma \|D \Bx\|, \kappa\}$ for some $D>0$. Then as $\|\Bx - \By\| \le 1/D^2$,  by the triangle inequality we have 

$$\dist(D \By, \Z^n) \le \min \{ \gamma (\|D \By\|+ D\|\Bx-\By\|), \kappa + D \|\Bx-\By\|\} \le  \min \{ (\gamma+\frac{1}{D}) \|D \By\|, \kappa + 1\} .$$

\end{proof}

There are two main advantages of working with unit vectors $\Bx=(x_1,\dots, x_n)$ of large LCD. First, as it turns out, if  $\LCD(\Bx)$ is large then the random sum $\xi_1 x_1+\dots+\xi_n x_n$, where $\xi_i$ are iid copies of $\xi$ from \eqref{eqn:subgaussian}, behaves like a continuous random variable of bounded density (even when the $\xi_i$ are discrete.) This statement is the content of the following result.

\begin{theorem}\cite{rv}\label{theorem:ball} For every   
 $$
  \eps \ge \frac{1}{\LCD_{\kappa,\gamma}(\Bx)}
  $$
  we have
  $$
 \sup_x \P(|\xi_1 x_1+\dots+ \xi_n x_n - x| \le \eps) =O\left(\frac{\eps}{\gamma} + e^{-\Theta(\kappa^2)}\right),
  $$

where the implied constants depend on $\xi$.
\end{theorem}

Second, vectors with small LCD can be well approximated by rational vectors $\Bp/\|\Bp\|$ with $\Bp\in \Z^n$ and $\|\Bp\|$ small.

\begin{theorem}\label{theorem:net}
Let $D\ge c \sqrt{n}$. Then the set $\{\Bx\in S^{n-1}: c\sqrt{n} \le \LCD_{\kappa,\gamma}(\Bx) \le D \}$ has a $(2\kappa/D)$-net $\CN_D$ of cardinality at most 

$$(C_0D/\sqrt{n})^n \log_2 D,$$

for some absolute constant  $C_0$.
\end{theorem}

To show this result, if suffices to establish nets for the level sets  $S_{D_0}:=\{x\in S^{m-1}: D_0 \le \LCD_{\kappa,\gamma}(x) \le 2D_0 \}$. 

\begin{lemma}\cite[Lemma 4.7]{rv-rec}\label{lemma:D_0}
There exists a $(2\kappa/D_0)$-net of $S_{D_0}$ of cardinality at most $(C_0D_0/\sqrt{n})^n$, where $C_0$ is an absolute constant.
\end{lemma}

Subdividing these nets into $(2\kappa/D)$-nets and taking the union as $D_0$ ranges over powers of two, we thus obtain Theorem \ref{theorem:net}. As the proof of Lemma \ref{lemma:D_0} is short and uses the important notion of $\LCD$, we include it here for the reader's convenience.

\begin{proof}(of Lemma \ref{lemma:D_0})
For $x\in S_{D_0}$, denote

$$D(x):= \LCD_{\kappa,\gamma}(x).$$

By definition, $D_0\le D(x)\le 2D_0$ and there exists $p\in \Z^m$ with 

$$\left\|x -\frac{p}{D(x)}\right\| \le \frac{\kappa}{D(x)} =O\left(\frac{n^{2c}}{n^{1-c}}\right) =o(1).$$ 

As $\|x\|=1$, this implies that $\|p\| \approx D(x)$, more precisely

\begin{equation}\label{eqn:LCD:net:1}
1- \frac{\kappa}{D(x)} \le \left\|\frac{p}{D(x)}\right\| \le  1+ \frac{\kappa}{D(x)}.
\end{equation}

This implies that 

\begin{equation}\label{eqn:LCD:net:2}
\|p\| \le (1+o(1)) D(x) <3D_0.
\end{equation}

It also follows from \eqref{eqn:LCD:net:1} that

\begin{equation}\label{eqn:LCD:net:3}
\left\|x -\frac{p}{\|p\|}\right\| \le \left\|x -\frac{p}{D(x)}\right\| + \left\| \frac{p}{\|p\|}(\frac{\|p\|}{D(x)} -1)\right\| \le 2\frac{\kappa}{D(x)} \le \frac{2\kappa}{D_0}.
\end{equation}

Now set 

$$\CN_0:=\left\{\frac{p}{\|p\|}, p\in \Z^m \cap B(0,3D_0)\right\}.$$

By \eqref{eqn:LCD:net:2} and \eqref{eqn:LCD:net:3}, $\CN_0$ is a $\frac{2\kappa}{D_0}$-net for $S_{D_0}$. On the other hand, it is known that the size of $\CN_0$ is bounded by $(C_0\frac{D_0}{\sqrt{m}})^m$ for some absolute constant  $C_0$.
\end{proof}

As Theorem \ref{theorem:ball} and Theorem \ref{theorem:net} suggest, we will choose $\CS$ to be the collection of unit vectors with large LCD: for $\gamma=1/2, \kappa=n^c$ and $D =\exp(n^c)$ with a sufficiently small constant $c$ to be chosen we set

\begin{equation}\label{eqn:CS}
\CS :=\{\Bx\in S^{n-1}, \LCD_{\gamma,\kappa}(\Bx) \ge D\}.
\end{equation}

We next show that this set contains "most" of the vectors of $S^{n-1}$. 

\begin{lemma}\label{lemma:dim1:dense} With $\kappa =n^c$ with some $c<1/6$ we have 
$$\Vol(S^{n-1}\backslash \CS) \le  \Vol_{n-1} (\BB(0, n^{-1/2+5c}))$$
\end{lemma}

\begin{proof}(of Lemma \ref{lemma:dim1:dense}) Let $\CN_D$ be one of the sets obtained from Theorem \ref{theorem:net}. Then by definition  
 
\begin{align*}
\Vol(S^{n-1}\backslash \CS) & \le \Vol_{n-1}(\CN_D + \BB(0,\kappa/D) \cap S^{n-1})   \le \frac{\pi^{(n-1)/2}}{\Gamma((n-1)/2)} (\kappa/D)^{n-1} |\CN_D|\\ 
&\le  \frac{\pi^{(n-1)/2}}{\Gamma((n-1)/2)} (C\kappa/D)^{n-1} \times (C_0D/\sqrt{n})^n \log_2 D\\
& \le   \frac{\pi^{(n-1)/2}}{\Gamma((n-1)/2)} (CC_0 \kappa/\sqrt{n})^n D^2 \\
& \le  \frac{\pi^{(n-1)/2}}{\Gamma((n-1)/2)} (n^{-1/2+4c})^n \\
& \le \Vol_{n-1} (\BB(0, n^{-1/2+5c})). 
\end{align*}
\end{proof}

Lemma \ref{lemma:dim1:dense} implies that for any $n^{-1/2+5c} \le \eps \le 1$, there exists a $\eps$-net $\CN_{start}$ of $S^{n-1}$ with size $(C/\eps)^n$ that belongs to $\CS$. This set $\CN_{start}$ will be the collection of our starting vectors $\Bx_0$ discussed in Section \ref{section:method}.

Now we will proceed to our main result of Step 1. For this, we will find the following lemma useful.

\begin{lemma}\cite[Lemma 2.2]{rv}\label{lemma:tensorization}  Let $\zeta_1,\dots,\zeta_n$ be independent non-negative random variables, and let $K, t_0 > 0$.  If one has
$$ \P( \zeta_k < t ) \leq K t$$
for all $k=1,\dots,n$ and all $t \geq t_0$, then one also has
$$ \P( \sum_{k=1}^n \zeta_k^2 < t^2 n ) \leq O((K t)^n)$$
for all $t \geq t_0$.
\end{lemma}

For our analysis below we recall the definition of $\CS$ from \eqref{eqn:CS}.

\begin{theorem}[key estimate, stability of non-structures]\label{theorem:stable} Assume that $A=(a_{ij})_{1\le i,j\le n}$ is a random matrix of size $n$ whose entries are iid copies of $\xi$ satisfying \eqref{eqn:subgaussian}. Let  $\Bx=(x_1,\dots,x_n)$ be any deterministic vector from $\CS$. Then 
$$\P_A(\frac{A\Bx}{\|A \Bx\|} \notin \CS) \le n^{-cn/8}.$$
\end{theorem}

\begin{proof}(of Theorem \ref{theorem:stable}) We first consider the event $\CE_1$ that $\|Ax\|^2 \le n^{1-c/2}$. As $\LCD_{\gamma, \kappa}(\Bx)\ge D = \exp(n^c) \gg \sqrt{n}$, by Theorem \ref{theorem:ball} 

$$\P(|a_{i1}x_1+\dots+a_{in} x_n| \le n^{-c/4}) =O(n^{-c/4}).$$

Thus by Lemma \ref{lemma:tensorization}

\begin{equation}\label{eqn:E1:stable}
\P(\CE_1) =\P(\|A\Bx\|^2 \le n^{1-c/2}) \le C^n n^{-cn/4}.
\end{equation}

Now for the event $\CE_2$ that $\|A\Bx\| \ge n^{1-2c}$, by  the standard Chernoff deviation result (as $\|A\Bx\|^2 =\sum_i (\sum_j a_{ij} x_j)^2$) we have

\begin{equation}\label{eqn:E3:stable}
\P(\|A\Bx\| \ge n^{1-2c}) \le \exp(-n^{1-4c}).
\end{equation}

On the complement of $\CE_1$ and $\CE_2$, for each $n^{1/2-c/4} \le r \le n^{1-2c}$ let us look at the event $\CE_r$ that $\|A\Bx-\By\| =O( r \kappa/D)$ for some vector $\By=(y_1,\dots,y_n)$ of norm $\|\By\|=r$ in the net $r \cdot \CN_D$ (with $\CN_D$ obtained from Theorem \ref{theorem:net}). Clearly this covers the event that $\LCD (Ax/\|Ax\|) \le D$ and $r-n^{-c}\le \|Ax\| \le r+n^{-c}$. 

Again, as $\LCD(\Bx)\ge D$ and that $r \kappa \ge n^{1/2 +3c/4} > n^{1/2}$, by Theorem \ref{theorem:ball}

\begin{align*}
\P(|a_{i1}x_1+\dots+a_{in} x_n - y_i| \le \frac{r \kappa}{D\sqrt{n}}) = O(\frac{r\kappa}{D\sqrt{n}}).
\end{align*}

Thus by Lemma \ref{lemma:tensorization},

\begin{align*}
 \P(\|A\Bx-\By\|^2 \le \frac{r^2 \kappa^2}{D^2})  \le (\frac{C r \kappa}{D\sqrt{n}})^n.
\end{align*}

We have thus obtained (taking into account of the size of $\CN_D$ from Theorem \ref{theorem:net})

\begin{align*}
\P\Big (\exists \By\in r \cdot \CN_D, \|A\Bx-\By\| \le O(r\kappa/D), \LCD (A\Bx/\|A\Bx\|) \le D \Big) & \le |\CN_D|  (\frac{C r\kappa}{D\sqrt{n}})^n  \le D^2 (\frac{C r \kappa}{n})^{n}\\
&\le n^{-c n},
\end{align*}

where we used the assumption $D = \exp(n^c)$ and $r \le n^{1-2c}$. 

Let $\CE_3$ be the event that $\LCD(A\Bx/\|A \Bx\|) \le D$ and $n^{1/2-c/4} \le \|A\Bx\| \le n^{1-2c}$. Then by taking any $\kappa/D^{O(1)}$-net $\{r_1,\dots, r_m\}$ of the segment $n^{1/2-c/4}\le r \le n^{1-2c}$ we have

\begin{align}\label{eqn:E2:stable}
\P(\CE_3) &\le \P\Big (\exists i, \exists \By \in r_i \cdot \CN_{D}, \|A\Bx-\By\|\le O(r_i\kappa/D), \LCD (A\Bx/\|A\Bx\|) \le D\Big) \nonumber \\
&\le ((n^{1-2c}-n^{1/2-c/4})D/\kappa) n^{-cn} \nonumber \\ 
&\le n^{-cn/2}.
\end{align}

The proof is complete by \eqref{eqn:E1:stable},  \eqref{eqn:E3:stable} and \eqref{eqn:E2:stable}.
\end{proof}

\section{Step 2 for (1) of Theorem \ref{theorem:main}: concentration of magnitude over non-structured vectors}\label{section:step2} 
Recall that
\begin{align*}
\frac{1}{N}\log \|B_N\Bx_0 \| = \frac{1}{N} \sum_{i=0}^{N-1} \log \|A_{i+1}\Bx_i \|,
\end{align*}

where 

$$\Bx_i =  \frac{A_{i} \dots A_2 A_1 \Bx_0}{\|A_{i} \dots A_2 A_1 \Bx_0\|}.$$

By Theorem \ref{theorem:stable}, we can assume that $\Bx_i \in \CS$ (i.e. $\LCD_{\gamma,\kappa}(\Bx_i) \ge D$) for all $1\le i\le N$ with a loss of $N \exp(-cn/8)$ in probability. 

In this section we study the concentration of $\log \|A_{i+1}\Bx_i\|$ around its zero mean (here again the randomness is with respect to $A_{i+1}$, conditioning on all $A_1,\dots, A_i$.)

Let $\Bx=(x_1,\dots,x_n)$ be a vector in $\CS$. Let $A=(a_{ij})_{1\le i\le n}$ be a random square matrix whose entries are iid copies of $\xi$ from \eqref{eqn:subgaussian}.  For short, set $\xi_i := a_{i1}x_1+\dots+ a_{in} x_n$ and 

$$y:=\log (\frac{1}{n}\|A \Bx\|_2^2) = \log (\frac{1}{n}(\xi_1^2+\dots+\xi_n^2)).$$

Before stating our estimates, we note that as $a_{ij}$ are subgaussian random variables of parameter $K$, so are the normalized random variables $\xi_1,\dots, \xi_n$. This implies that $\xi_i^2$ are exponential random variables (since $\P(\xi_i^2 \ge t) =\P(|\xi_i| \ge \sqrt{t})  \le O(\exp(-t/K))$). As a consequence, for any $x\ge 0$ (see for instance \cite[Proposition 5.16]{V-note})

\begin{equation}\label{eqn:Bernstein}
\P(|\xi_1^2+\dots + \xi_n^2-n| \ge nx) \le 2 e^{-c\min\{ nx^2/K^2, nx/K\}}.
\end{equation}

\begin{theorem}[concentration over non-structured vectors] \label{theorem:deviation} We have

\begin{enumerate}[(i)]
\item for any $t>0$,
$$\P(y\ge t) \le e^{-c't^2n}$$
for some absolute constant $c'>0$;
\item for any $0\le t\le 2\log D$
$$\P(y\le -t) \le \min\{(K e^{-t/2})^n,1\}$$
where $K$ is the parameter from \eqref{eqn:subgaussian};
\item for any $t\le O(1)$
$$\P(|y|\ge t) \le e^{-c''t^2n}$$
for some absolute constant $c''>0$.
\end{enumerate}

\end{theorem}

\begin{proof}(of Theorem \ref{theorem:deviation}) For $(i)$, with the parameter $\lambda=ctn/2K^2+1$ 

\begin{equation}\label{eqn:y:1}
\P(y\ge t) = \P(e^{\lambda y} \ge e^{\lambda t}) \le e^{-\lambda t} \E ((\frac{1}{n}(\xi_1^2+\dots +\xi_n^2))^\lambda) =  e^{-\lambda t} \lambda \int_0^{\infty} x^{\lambda-1} \P(z>x) dx,
\end{equation}

where $z := \frac{1}{n} (\xi_1^2+\dots+\xi_n^2)$.  Note that  we can trivially bound 

\begin{equation}\label{eqn:y:2}
\int_0^{1} x^{\lambda-1} \P(z>x) dx\le 1/\lambda.
\end{equation}

For the integral corresponding to $x\ge 1$ we use \eqref{eqn:Bernstein}

$$\int_1^{\infty} x^{\lambda-1} \P(z>x) dx =  \int_0^{\infty} (1+x)^{\lambda-1} \P(z>x+1) dx \le 2 \int_0^{\infty} (1+x)^{\lambda-1}   e^{-c \min\{ nx^2/K^2, nx/K\}} dx.
$$

To this end, 

\begin{equation}\label{eqn:y:3}\int_0^{1} (1+x)^{\lambda-1}   e^{-\frac{c}{K^2} nx^2}  dx\le \int_0^{1}  e^{-\frac{c}{K^2} nx^2 + (\lambda-1)x}dx \le 1 +  e^{ K^2(\lambda-1)^2/cn} \le e^{\lambda t/2}.
\end{equation}

Furthermore, for $t< 2$ we have

\begin{equation}\label{eqn:y:4}
\int_1^{\infty} (1+x)^{\lambda-1}   e^{-\frac{c}{K^2} nx}  dx\le \int_1^{\infty}  e^{x \frac{c}{2K^2}n(t-2)}   dx  =O(1).
\end{equation}

For $t\ge 2$, let $x'$ be such that $(\lambda-1)/n= x'/\log (1+x')$ (thus $x' \asymp t \log t$) then

\begin{equation}\label{eqn:y:5}
\int_1^{\infty} (1+x)^{\lambda-1}   e^{-\frac{c}{K^2} nx}  dx\le \int_1^{x'} (1+x)^{\lambda-1}   dx  +1 \le e^{\lambda t/2}.
\end{equation}

Combining \eqref{eqn:y:1},\eqref{eqn:y:2},\eqref{eqn:y:3}, \eqref{eqn:y:4} and \eqref{eqn:y:5} we obtain 

$$\P(y\ge t) \le e^{-\lambda t/2} \le  e^{-c' t^2n}.$$


For the lower tail (ii),  we recall that $\LCD_{\gamma,\kappa}(\Bx) \ge D$. By Lemma \ref{lemma:tensorization}, for any $t$ such that $e^{-t/2}\ge 1/D$

$$\P(y\le -t) = \P(e^y \le e^{-t}) = \P(\xi_1^2+\dots+\xi_n^2 \le n e^{-t}) \le (K e^{-t/2})^n .$$

For (iii), we need to estimate $\P(y \le -t)$ with $0\le t =O(1)$. We have

\begin{align*}
\P(y\le -t )=\P(z \le e^{-t}) \le \P(z \le 1-\min\{t, 1\}/2) &\le \P(|z-1| \ge \min\{t,1\}/2)\\ 
&\le 2 e^{-\frac{c''}{K^2} n t^2}
\end{align*}

where we used \eqref{eqn:Bernstein} in the last estimate.  

\end{proof}

\section{Step 3 for (1) of Theorem \ref{theorem:main}: concluding the proof}\label{section:step3}

Let $\Bx_0$ be any vector from $\CN_{start}$. We will show

\begin{lemma}\label{lemma:top}
For any $t\ge 1/n$ we have 

$$\P\Big( |\frac{1}{N}  \log \|B_N \Bx_0\| | \ge t \Big) \le \exp(-c\min\{t^2, t\} Nn^{1-2c}) + N n^{-cn}.$$
\end{lemma}

It is clear that Theorem \ref{theorem:main} follows from Lemma \ref{lemma:top} after taking union bound over $\CN_{start}$.

\begin{proof}(of Lemma \ref{lemma:top}) First, by Theorem \ref{theorem:stable}, the event $\CF_1$ that $\Bx_i \in \CS$ for all $1\le i\le N$ holds with probability 

$$\P(\CF_1) \ge 1 - N n^{-cn}.$$

Consider the random sum 

$$S =\frac{1}{N} (y_1+\dots+y_N),$$

where  $\By_i = \log \|A_{i+1} \Bx_i\|$.

Basing on Theorem \ref{theorem:deviation}, the event $\CF_2$ such that $|y_i|\le 2 \log D$ for all $y_i, 1\le i\le N$ satisfies 

$$\P(\CF_2) \ge 1-N D^{-n}.$$

Introduce the new random variables $y_i':= y_i 1_{|y_i| \le 2 \log D}$ and $y_i'' := y_i' - \E_{A_i} y_i'$. As customary, in what follows our probability is with respect to $A_{i}$, conditioning on $A_1,\dots, A_{i-1}$. By (iii) of Theorem \ref{theorem:deviation}, for $|t|=O(1)$ 

\begin{equation}\label{eqn:y':1}
\P(|y_i'| \ge t) \le \P(|y_i|\ge t) \le e^{-ct^2n}.
\end{equation}

Also, by (i) and (ii) of Theorem \ref{theorem:deviation}, for $O(1) \le t\le 2 \log D$

\begin{equation}\label{eqn:y':2}
\P(|y_i'| \ge t) \le \P(|y_i|\ge t) \le K^n e^{-tn/2} + e^{-ct^2n}.
\end{equation}

Consequently,

$$|\E_{A_i} y_i'| \le \int_{0}^{2 \log D} t \P(|y_i'|\ge t) \le O(\int_{0}^{1/\sqrt{n}} t dt) =O(1/n).$$






 Next consider the martingale sum

$$S'':= \frac{1}{N}(y_1''+\dots+y_N'').$$

By definition, $|y_i''|\le 2 \log D +O(1/n)< 3 \log D$. Also by \eqref{eqn:y':1} and \eqref{eqn:y':2}, for $t\ge 1/n$ 

$$\P(|y_i''|\ge t) \le \P(|y_i'|\ge t) \le  \exp(-c\min\{t^2, t\} n).$$ 

This implies the following conditional estimate for $\lambda = c tn$

$$e^{-2\lambda t} \E (e^{\lambda y_i'' }|A_1,\dots,A_{i-1}), e^{-2\lambda t} \E (e^{-\lambda y_i'' }|A_1,\dots,A_{i-1}) \le \exp(-c \min \{t, t^2\}n).$$

Following the proof of Azuma's martingale concentration, for $t\ge 1/n$

\begin{align*}
\P(S''\ge 2t) = \P(y_1''+\dots+y_N'' \ge 2Nt) &=  \P(\exp((y_1''+\dots+y_N'')\lambda) \ge \exp(2\lambda Nt )) \\
&\le  \exp(-2\lambda Nt)\E (\exp((y_1''+\dots+y_N'')\lambda)\\
&\le \exp(-c'\min\{t^2, t\}Nn).
\end{align*}

We also obtain the same bound for $\P(S'' \le -2t)$. Thus, as $\P(|S| \ge 2t) \le \P(|S| \ge2 t | \CF_1 \cap \CF_2) \P(\CF_1 \cap \CF_2) + \P(\bar{\CF_1}  \cup \bar{\CF_2})$, we have

\begin{align*}
\P(|S| \ge 2t +O(1/n)) \le   \P(|S''| \ge 2t) + \P(\bar{\CF_1}  \cup \bar{\CF_2})  &\le \exp(-c\min\{t^2, t\}Nn) +\P(\bar{\CF_1}) + \P(\bar{\CF_2})\\
&\le \exp(-c\min\{t^2, t\}Nn) + N  n^{-cn}.
\end{align*}

\end{proof}

\section{Proof of (2) of Theorem \ref{theorem:main} with the modified three-step plan}\label{section:second} Although our treatment here is analogous to that of the top exponent, the argument is far more complicated as we have to take care of the angles of the pair vectors. 

We will first introduce some additional structures. The definition of LCD can be naturally extended to joint structure of two vectors. Let $\gamma,\kappa$ be given parameters and let $\Bx,\By$ be two vectors. Define

$$ \LCD_{\gamma,\kappa}(\Bx,\By):=\inf_{\theta_1^2+\theta_2^2=1} \LCD_{\gamma,\kappa}(\theta_1 \Bx+\theta_2 \By).
$$

For the rest of this section we will choose $\gamma$ to be a sufficiently small and $\kappa=n^c$ for some constant $c<1/16$.

\subsection{Step 1} Set $D=\exp(n^c)$. First of all we will have to choose $\CP_{start} \subset S^{n-1} \times S^{n-1}$ to satisfy the following conditions

\begin{itemize}
\item for all $(\Bx',\By')\in \CP_{start}$ we have $\|\Bx' \wedge \By'\|\ge \eps$, 
\vskip .1in
\item  for all $(\Bx',\By')\in \CP_{start}$ we have 
\begin{equation}\label{eqn:CP}
\LCD_{\kappa,\gamma}(\Bx'/\|\Bx'\wedge \By'\|,\By'/\|\Bx'\wedge \By'\|) \ge D,
\end{equation}
\vskip .1in
\item  for any $(\Bx, \By) \in S^{n-1} \times S^{n-1}$ there exists $(\Bx',\By') \in \CP_{start}$ such that $\|\Bx-\Bx'\|, \|\By-\By'\|\le \eps$.
\vskip .1in

\end{itemize}

Remark that a direct choice of $\CN_{start} \times \CN_{start}$ (with $\CN_{start}$ from Section \ref{section:step1}) would not work because there were no information on the joint structure.

\begin{lemma}\label{lemma:P-start} There exists a set $\CP_{start}$ with the above properties.
\end{lemma}

\begin{proof}(of Lemma \ref{lemma:P-start}) In what follows we will be focusing on the set $\CS_{separate}$ of pairs of unit vectors $\Bx,\By \in S^{n-1}$ with $\|\Bx \wedge \By\| \ge \eps$.

Assume that $(\Bx,\By) \in \CS_{separate}$ which violates \eqref{eqn:CP}. In other words there exist  $\theta_1^2 +\theta_2^2=1$ with

\begin{equation}\label{eqn:P-start:1}
\LCD( \frac{1}{\|\Bx \wedge \By\|} \theta_1\Bx + \frac{1}{\|\Bx \wedge \By\|}  \theta_2 \By)\le D.
\end{equation}

In the next step we $1/D^{O(1)}$-approximate the parameters $\theta_1,\theta_2, \|\Bx \wedge \By\|$ by numbers of form $k \frac{1}{D}, k\in \Z$. Thus by losing a factor of $D^{O(1)}$ in probability at most, by using Fact \ref{fact:approxLCD} and that $\|\Bx \wedge \By\|\ge \eps$, we can treat $\theta_1,\theta_2, \|\Bx \wedge \By\|$ as constants. Furthermore, by changing the vector direction if needed, without loss of generality we can assume $\theta_2 \ge \theta_1>0$. 

By Theorem \ref{theorem:net}, we thus have three vectors $\Bx, \By, \Bz$ where $\Bz = t_1 \Bx + t_2 \By$ with $\Bz \in k \cdot \CN_D$ and $t_1^2+ t_2^2 \gg 1$ as well as $t_2\ge t_1 \ge 0$.

 Solving for $\By$, 

$$\By = \frac{1}{t_2}\Bz - \frac{t_1}{t_2} \Bx.$$

We conclude that there exists an absolute constant $C$ such that for any given $\Bx \in S^{n-1}$, the vectors $\By \in S^{n-1}$ for which \eqref{eqn:P-start:1} holds belong to a set $\CS_\Bx$ of volume at most

\begin{align*}
\Vol(\CS_\Bx) \le \Vol_{n-1}(B(0,C\kappa/D)) \times |\CN_D| \times D^{O(1)}. 
\end{align*}

where the first two factors come from the approximation of $\Bz$ and the magnifying factor $t_1/t_2$,  while the third factor comes from approximations of the parameters $\theta_1,\theta_2, \|\Bx \wedge \By\|$ by numbers of the form $k \frac{1}{D}, k\in \Z$ as above. 

Varying $\Bx \in S^{n-1}$, the total volume $\Vol_T$ of such a pair $(\Bx, \By) \in \CS_{separate}$ satisfying \eqref{eqn:P-start:1} is at most  

\begin{align}\label{eqn:dim2:rare}
\Vol_T &\le \Vol_{n-1}(S^{n-1})  \Vol_{n-1}(\BB(0,C\kappa/D)) \times |\CN_D| \times D^{O(1)} \nonumber \\
 &\le \Vol_{n-1}(S^{n-1})  \frac{\pi^{(n-1)/2}}{\Gamma((n-1)/2)} (C\kappa/D)^{n-1} \times (C_0D/\sqrt{n})^n \log_2 D \times D^{O(1)} \nonumber \\
 &\le \Vol_{n-1}(S^{n-1})  \frac{\pi^{(n-1)/2}}{\Gamma((n-1)/2)} (CC_0 \kappa/\sqrt{n})^n D^{O(1)} \nonumber \\
 & \le \Vol_{n-1}(S^{n-1})  \times \Vol_{n-1}(S^{n-1})  \times  (CC_0 \kappa/\sqrt{n})^n D^{O(1)}.
\end{align}
 
Next, notice that the total volume of an $\eps/2$-neighborhood of any point on $S^{n-1} \times S^{n-1}$ is at least 
 
\begin{align}\label{eqn:dim2:rich}
V_{\eps/2} &\ge \Vol_{n-1} (\BB(0,\eps/2)\cap S^{n-1}) \times \Vol_{n-1}(\BB(0,\eps/2) \cap S^{n-1}) \nonumber \\
&\ge  \Vol_{n-1}(S^{n-1})  \times \Vol_{n-1}(S^{n-1})  \times  (\eps/2C)^{2n}.
\end{align}

Thus $V_{\eps/2} > \Vol_T$ if we choose $\kappa = n^c$ for some $c<1/16$. 

It follows that for any $\eps/2$-neighborhood of any point on $S^{n-1} \times S^{n-1}$, there exists a point $(\Bx,\By) \in \CS_{separate}$  such that  \eqref{eqn:CP} holds. 
The proof of Lemma \ref{lemma:P-start} is then complete by considering a maximal $\eps/4$-packing of $\CS_{separate}$.
\end{proof}

Let $\CP:=\CP_D$ be the collection of vector pairs $(\Bx,\By)$ in $\R^n \times \R^n$ such that 

$$\LCD_{\kappa,\gamma}(\frac{\Bx}{\|\Bx \wedge \By\|},\frac{\By}{\|\Bx \wedge \By\|})\ge D.$$
 
Note that here the vectors of $\CP$ do not need to be unit, and by definition $\CP_{start} \subset \CP \cap S^{n-1} \times S^{n-1}$. Our next key result is an analog of Theorem \ref{theorem:stable} for joint-structures.

\begin{theorem}[stability of non-structures, jointly]\label{theorem:stable:2} Assume that $A=(a_{ij})_{1\le i,j\le n}$ is a random matrix of size $n$ whose entries are iid copies of $\xi$ satisfying \eqref{eqn:subgaussian}. Let  $(\Bx, \By)$ be any deterministic vector pair from $\CP$. Then 
$$\P_A\Big((\frac{A\Bx}{\|A \Bx\|}, \frac{A\By}{\|A \By \|})  \notin \CP\Big) \le n^{-cn/8}.$$
\end{theorem}

By Theorem \ref{theorem:dim2:upper} (to be proved separately later), with a probability at least $1- n^{-cn/16}$ we can assume that 

$$\| A\Bx/\|A \Bx\| \wedge A\By/\|A \By \| \| \le n^{c/16} \|\Bx \wedge \By\|.$$
 
Thus by definition of LCD (see Fact \ref{fact:approxLCD}), for Theorem \ref{theorem:stable:2} it suffices to show 
 
 \begin{equation}\label{eqn:stable:2'}
 \P_A\Big(\LCD_{\gamma, \kappa}(\frac{A\Bx/\|A \Bx\|}{\| \Bx \wedge \By \|},\frac{A\By/\|A \By\|}{\| \Bx \wedge \By \|})  \le D'\Big) \le n^{-cn/8},
 \end{equation}

where $D'=Dn^{c/16}$.

\begin{proof}(of Theorem \ref{theorem:stable:2}) By the proof of Theorem \ref{theorem:stable}, it suffices to focus on the event $\CE_1$ that $\|A\Bx\|^2, \|A\By\|^2 \ge n^{1-c/2}$. 

For each $n^{1/2-c/4} \le r,s \le n^{1-2c}$ (which can be approximated by integral points of the form $r_i=i n^{-c}, i\in \Z$) let us look at the event $\CE_{r_i,s_j} $ that $r_i \le \|A\Bx\|  \le r_{i+1}$ and $s_j \le \|A\By\|\le s_{j+1}$   and such that

\begin{equation}\label{eqn:stable:2:LCD}
\LCD_{\gamma,\kappa}( \frac{1}{\|\Bx \wedge \By\|} A\Bx/r_i,   \frac{1}{\|\Bx \wedge \By\|}  A\By/s_j) \le (1+o(1))D'.
\end{equation}

(Note that by Fact \ref{fact:approxLCD} $\LCD_{\gamma,\kappa}( \frac{1}{\|\Bx \wedge \By\|}  A\Bx/r,  \frac{1}{\|\Bx \wedge \By\|}  A\By/s)$ are comparable in the range $r_i \le r\le r_{i+1}, s_j \le s \le s_{j+1}$.)

In other words, by definition of joint LCD there exist $\theta_1,\theta_2$ with $\theta_1^2+\theta_2^2=1$ such that 

\begin{equation}\label{eqn:stable:2:approx}
\LCD_{\gamma,\kappa}( \frac{1}{\|\Bx \wedge \By\|} \theta_1 A\Bx/r_i + \frac{1}{\|\Bx \wedge \By\|}  \theta_2 A\By/s_j) \le (1+o(1))D'.\end{equation}

We can write 

$$\theta_1 A\Bx/r_i + \theta_2 A\By/s_j =  \theta A (\theta_1'\Bx +\theta_2' \By).$$ 

Here

$$ \theta= \sqrt{(\theta_1/r_i)^2 + (\theta_2/s_j)^2} \mbox{ and } \theta_1'=\frac{\theta_1/r_i}{\theta}, \theta_2'=\frac{\theta_2/s_j}{\theta}.$$  

Because $\theta$ is not too small ($\theta \approx n^{1/2}$), we can again assume it to have the form $i D^{-O(1)}$ and relax the constrain $\theta_1'^2+\theta_2'^2=1$ to $1- D^{-O(1)} \le \theta_1'^2+\theta_2'^2 \le 1+D^{-O(1)}$.



Now we look at the event $\LCD_{\gamma,\kappa}( \frac{1}{\|\Bx \wedge \By\|}  \theta A (\theta_1'\Bx +\theta_2' \By)) \le (1+o(1))D'$ from \eqref{eqn:stable:2:approx}. By Theorem \ref{theorem:net}, there exists $\Bu \in \theta^{-1}\cdot \CN_{D'}$ such that 

\begin{equation}\label{eqn:dim2:theta'}
\| \frac{1}{\|\Bx \wedge \By\|}  A (\theta_1'\Bx +\theta_2' \By)-\Bu\| \le O( \theta^{-1} \kappa/D').
\end{equation}

By passing to numbers of the form $i (\sqrt{n}D')^{-1}, i\in \Z$, up to a multiplicative factor in the RHS of \eqref{eqn:dim2:theta'}, we can assume $\theta_1'$ and $\theta_2'$ to be fixed so that we can take a union bound over the set of these integral points, which obviously has cardinality $D^{O(1)}$).  In other words, by passing to those approximated points, with a loss of a factor of $n^{O(1)}D^{O(1)}$ in probability we will be arriving at \eqref{eqn:dim2:theta'} with fixed $\theta_1',\theta_2'$ and $\Bx,\By,\Bu$.

Now we analyze the probability of the event from \eqref{eqn:dim2:theta'} by  invoking the argument from the proof of Theorem \ref{theorem:stable}. As $(\Bx,\By) \in \CP$, we have $\LCD( \frac{1}{\|\Bx \wedge \By\|}\Bx, \frac{1}{\|\Bx \wedge \By\|}\By)\ge D$, henceforth 

$$\LCD( \frac{1}{\|\Bx \wedge \By\|} \theta_1'\Bx + \frac{1}{\|\Bx \wedge \By\|} \theta_2' \By) \ge D.$$ 

Note that $\theta^{-1} \kappa \ge n^{1/2 +3c/4} > n^{1/2}$. As $\|\theta_1' \Bx + \theta_2' \By\| =\Omega( \|\Bx \wedge \By\|)$ for any $\theta_1', \theta_2'$ with $\theta_1'^2 +\theta_2'^2 =1+o(1)$, by Theorem \ref{theorem:ball}

\begin{align*}
\P(|a_{i1}z_1+\dots+a_{in} z_n - u_i| \le  \frac{r \kappa}{D\sqrt{n}} ) = O(\frac{r\kappa}{D\sqrt{n}}),
\end{align*}

where for short we set $\Bz=(z_1,\dots,z_n):= \theta_1' \Bx + \theta_2' \By$. 

Thus by Lemma \ref{lemma:tensorization}, for any fixed $\Bu$

\begin{align*}
 \P(\| \frac{1}{\|\Bx \wedge \By\|}  A\Bz-\Bu\|^2 \le  \frac{r^2 \kappa^2}{D^2})  \le (\frac{C r \kappa}{D\sqrt{n}})^n.
\end{align*}

We have thus obtained (taking into account of the size of $\CN_{D'}$ from Theorem \ref{theorem:net})

\begin{align*}
& \P\Big (\exists \Bu\in \theta^{-1} \cdot \CN_D, \| \frac{1}{\|\Bx \wedge \By\|}  A\Bz-\Bu\| \le O(r\kappa/D), \LCD ( \frac{1}{\|\Bx \wedge \By\|} A\Bz) \le (1+o(1))D \Big) \\
& \le |\CN_{D'}|  (\frac{C r\kappa}{D\sqrt{n}})^n \le n^{-cn}.
\end{align*}
To complete our proof, we take the union bound over all the choices of $r_i, s_j, \theta, \theta_1', \theta_2'$ to obtain a bound $n^{-cn} D^{O(1)} \le n^{-cn/2}$, completing the proof of \eqref{eqn:stable:2'}.
\end{proof}

\subsection{Step 2} Now we turn to the second step of the plan to control $\frac{\|A\Bx/\|A\Bx\| \wedge A \By/\|A\By\|\|}{\|\Bx \wedge \By\|}$ for given $(\Bx,\By)\in \CP$. Note that if $\|\Bu\|=\|\Bv\|=1$, then 

\begin{align}\label{eqn:dim2:wedge}
\|\Bu \wedge \Bv\|^2= \|\Bu\|^2\|\Bv\|^2-\langle \Bu, \Bv\rangle^2 = 1 - \langle \Bu, \Bv\rangle^2 & = 1-(1-\|\Bu -\Bv\|^2/2)^2 \nonumber \\ 
&= \|\Bu-\Bv\|^2 - (\frac{\|\Bu -\Bv\|^2}{2})^2.
\end{align}

We will be finding the following fact useful.

\begin{fact}\label{fact:f} Let $f(x) =x - x^2/4, 0\le x \le 1$, and let $0<x_1, x_2<1$. Let $t>0$ be a parameter.

\begin{itemize}
\item If $f(x_2)/f(x_1) \ge 1+t $ then $x_2/x_1 \ge 1+t$.
\vskip .1in
\item If $f(x_2)/f(x_1) \le t <1$ then $x_2/x_1 \le t$.
\end{itemize}
\end{fact}

To bound $\frac{\|A\Bx/\|A\Bx\| \wedge A \By/\|A\By\|\|}{\|\Bx \wedge \By\|}$ from below, we invoke the following result.

\begin{theorem}\label{theorem:dim2:lower} For any $(\Bx,\By)\in \CP$ and any $\delta \ge 1/D$,

$$\P_A(\frac{\|A\Bx/\|A\Bx\| \wedge A \By/\|A\By\|\|}{\|\Bx \wedge \By\|} \le \delta^2) \le (C\delta)^n.$$

\end{theorem}

As by Theorem \ref{theorem:deviation} $\|A\Bx\|, \|A \By\| \ge \delta \sqrt{n}$ with probability at least $1- (C\delta)^n$, we obtain

\begin{corollary}\label{cor:dim2:lower} 
$$\P_A( \frac{\|A\Bx \wedge A \By\|}{\|\Bx \wedge \By\|} \le n \delta^4) \le 2(C\delta)^n.$$
\end{corollary}

\begin{proof}(of Theorem \ref{theorem:dim2:lower}) By \eqref{eqn:dim2:wedge}, the assumption $\|\frac{A\Bx}{\|A \Bx\|} \wedge \frac{A\By}{\|A \By\|}\| \le \delta^2 \|\Bx \wedge \By \|$ implies that 

$$\|A(\frac{\Bx}{\|A \Bx \|} - \frac{\By}{\|A \By\|})\|\ll \delta^2 \|\Bx \wedge \By\|.$$

Again by Theorem \ref{theorem:deviation}, with probability at least $1- (C\delta)^n$ we can assume that 

$$C\delta \sqrt{n} \le \|A\Bx\|, \|A \By\| \le \delta^{-1} \sqrt{n}.$$ 

Thus the event  $\|A(\frac{\Bx}{\|A \Bx \|} - \frac{\By}{\|A \By\|})\| \le \delta^2 \|\Bx \wedge \By\|$ implies that $\|\frac{1}{\sqrt{n}}A (\alpha_1 \Bx + \alpha_2 \By)\| \ll \delta \|\Bx \wedge \By\|$ with some coefficients  $\alpha_1,\alpha_2 $ satisfying $\alpha_1^2 + \alpha_2^2=1$. 

In what follows we will consider this event. Notice that 

\begin{equation}\label{eqn:dim2:u}
\|\alpha_1 \Bx + \alpha_2\By\|^2 \ge 1- |\langle \Bx,\By \rangle| \ge \frac{1}{2}(1- |\langle \Bx,\By \rangle|^2) \ge \frac{1}{2}\|\Bx \wedge \By\|^2.
\end{equation}

Now we pass to consider a $1/D^{O(1)}$-net $\CM$ with respect to $(\theta_1,\theta_2)$ over the ellipsoid  $\theta_1 \Bx/\|\Bx \wedge \By\| +\theta_2 \By/\|\Bx \wedge \By\|, \theta_1^2 +\theta_2^2=1$. As this set is one-dimensional, one can take $|\CM| =D^{O(1)}$. 

Because we can assume that $\|A\| = O(\sqrt{n})$,

$$\inf_{\alpha_1^2 + \alpha_2^2=1} \|\frac{1}{\sqrt{n}}A (\alpha_1 \Bx/\|\Bx\wedge \By\| + \alpha_2 \By/\|\Bx \wedge \By\|)\| \ge \inf_{\Bu \in \CM} \|\frac{1}{\sqrt{n}}A \Bu \| -O(D^{-O(1)}).$$

But $\delta \ge 1/D$, we conclude that the event  $\|A(\frac{\Bx}{\|A \Bx \|} - \frac{\By}{\|A \By\|})\| \le \delta^2 \|\Bx \wedge \By\|$ implies the event $\CE$ where 

$$\CE:= \{\inf_{\Bu \in \CM} \|\frac{1}{\sqrt{n}}A \Bu \| \le 2\delta\}.$$

To estimate this event, choose any point $\Bu$ from $\CM$. As $(\Bx,\By)\in \CP$, we have $\LCD(\Bx/\|\Bx \wedge \By\|, \By/\|\Bx \wedge \By\|)\ge D$, and so we also have 

$$\LCD(\Bu) \ge D.$$

By Theorem \ref{theorem:deviation} and by \eqref{eqn:dim2:u}, as $\|\Bu\|\ge 1/\sqrt{2}$, as long as $\delta \gg 1/D$ we have

$$\P(\|\frac{1}{\sqrt{n}}A \Bu\| \le (C\delta)^n.$$

Thus 

$$\P(\CE) \le D^{O(1)} (C\delta)^n \le (C' \delta)^n.$$
\end{proof}

In our next theorem we give an analog of Theorem \ref{theorem:deviation}.

\begin{theorem}\label{theorem:dim2:upper} There exist constants $C,c$ such that for any $t>0$ and any $(\Bx,\By)\in \CP$ we have

\begin{enumerate}[(i)]
\item $$\P_A(\frac{\|A\Bx/\|A\Bx\| \wedge A \By/\|A\By\|\|}{\|\Bx \wedge \By\|} \ge C) \le \exp(-cn).$$
\item $$\P_A(\frac{\|A\Bx \wedge A \By\|}{\|\Bx \wedge \By\|} \ge n e^t) \le e^{-c t^2 n}.$$
\item Furthermore, if $t=o(\log n)$ then one also has 

$$\P_A(\frac{\|A\Bx \wedge A \By\|}{\|\Bx \wedge \By\|} \le n e^{-t})  =O(K^n e^{-tn/2}+e^{-c'' t^2 n}).$$
\end{enumerate}
\end{theorem}

\begin{proof}(of Theorem \ref{theorem:dim2:upper})  First we prove (i). By Fact \ref{fact:f} the assumption $\|\frac{A \Bx}{\|A \Bx\|} \wedge \frac{A\By}{\|A \By\|}\| \ge C \|\Bx \wedge \By \|$ implies that 

$$\|A(\frac{\Bx}{\|A \Bx \|} - \frac{\By}{\|A \By\|})\|\ge C \|\Bx-\By\|.$$

Note that 

$$\frac{\|A(\frac{\Bx}{\|A \Bx \|} - \frac{\By}{\|A \By\|})\|}{\|\Bx-\By\|} =\frac{\|A(\frac{\Bx}{\|A \Bx \|} - \frac{\By}{\|A \By\|})\|}{\|\frac{\Bx}{\|A \Bx \|} - \frac{\By}{\|A \By\|}\|} \frac{\|{\frac{\Bx}{\|A \Bx \|} - \frac{\By}{\|A \By\|}\|} }{\|\Bx -\By\|}.$$

Here, by (i) and (ii) of Theorem \ref{theorem:deviation} the following holds with probability at least $1- \exp(-cn)$ 

\begin{align*}
\frac{\|\frac{\Bx}{\|A \Bx \|} - \frac{\By}{\|A \By\|}\| }{\|\Bx -\By\|} &\le \frac{\|\frac{\Bx-\By}{\|A \Bx \|} +\By (1/\|A\Bx\| -1/\|A \By\|) \|}{\|\Bx -\By\|} \\
&\le \frac{1}{\|A \Bx \|} +  \frac{|\|A\Bx\| - \|A \By\||}{ \|A \Bx\| \|A \By\| \|\Bx -\By\|}\\
&\le \frac{1}{\|A \Bx \|} +  \frac{\|A(\Bx-\By)\|}{ \|A \Bx\| \|A \By\| \|\Bx -\By\|} \\
& \le \frac{1}{\|A \Bx \|} +  \frac{\|A\|}{ \|A \Bx\| \|A \By\|} \le \frac{C'}{\sqrt{n}}.
\end{align*}

Thus the event  $\|A(\frac{\Bx}{\|A \Bx \|} - \frac{\By}{\|A \By\|})\|\ge C \|\Bx-\By\|$ implies that there exists $\Bz$ such that $\|A \Bz\|/\|\Bz\|\ge C/C'$, and this holds with probability $\exp(-cn)$ if $C/C'$ is sufficiently large.

Now we prove (ii). By changing the size of $\Bx$ or $\By$ when necessary, without loss of generality we assume 

$$-1< \langle \Bx,\By \rangle 
\le 0.$$ 

We write

\begin{align*}
\frac{\|A\Bx \wedge A \By\|}{\|\Bx \wedge \By\|} = \|A \Bx\| \|A\By\| \frac{\|A\Bx/\|A\Bx\| \wedge A \By/\|A \By\|\|}{\|\Bx \wedge \By\|}.
\end{align*}

Thus our assumption implies that

\begin{align*}
\|A\Bx/\|A\Bx\| \wedge A \By/\|A \By\|\| \ge \frac{n e^t}{ \|A\Bx\| \|A\By\|}{\|\Bx \wedge \By\|}.
\end{align*}

By Fact \ref{fact:f}, we then have

\begin{equation}\label{eqn:dim2:upper:1}
\|A(\Bx/\|A\Bx\| -\By/\|A \By\|)\| \ge \frac{n e^t}{ \|A\Bx\| \|A\By\|}{\|\Bx - \By\|}.
\end{equation}

Now we argue as in the proof of (i),

\begin{align}\label{eqn:dim2:upper:2}
&\frac{\|A(\frac{\Bx}{\|A \Bx \|} - \frac{\By}{\|A \By\|})\|}{\|\Bx-\By\|} =\frac{\|A(\frac{\Bx}{\|A \Bx \|} - \frac{\By}{\|A \By\|})\|}{\|\frac{\Bx}{\|A \Bx \|} - \frac{\By}{\|A \By\|}\|} \frac{\|{\frac{\Bx}{\|A \Bx \|} - \frac{\By}{\|A \By\|}}\|}{\|\Bx -\By\|} \nonumber \\
&=\frac{\|A(\frac{\Bx}{\|A \Bx \|} - \frac{\By}{\|A \By\|})\|}{\|\frac{\Bx}{\|A \Bx \|} - \frac{\By}{\|A \By\|}\|} \sqrt{1/\|A\Bx\|^2 + 1/\|A\By\|^2} \frac{\frac{1}{\sqrt{1/\|A\Bx\|^2 + 1/\|A\By\|^2}}\|{\frac{\Bx}{\|A \Bx \|} - \frac{\By}{\|A \By\|}}\|}{\|\Bx -\By\|}.
\end{align}

Note that as $\langle \Bx, \By \rangle \le 0$, for any $\alpha^2 +\beta^2=1$, 

$$\||\alpha| \Bx - |\beta| \By\|^2 = 1 - 2 |\alpha \beta| \langle \Bx,\By \rangle \le  1 - \langle \Bx,\By \rangle = \frac{1}{2}\|\Bx -\By\|^2.$$

Thus

\begin{equation}\label{eqn:dim2:upper:3}
\frac{\frac{1}{\sqrt{1/\|A\Bx\|^2 + 1/\|A\By\|^2}}\|{\frac{\Bx}{\|A \Bx \|} - \frac{\By}{\|A \By\|}}\|}{\|\Bx -\By\|} \le \frac{1}{\sqrt{2}}.
\end{equation}

It follows from \eqref{eqn:dim2:upper:1}, \eqref{eqn:dim2:upper:2} and \eqref{eqn:dim2:upper:3} that

\begin{equation}
\frac{\|A(\frac{\Bx}{\|A \Bx \|} - \frac{\By}{\|A \By\|})\|}{\|\frac{\Bx}{\|A \Bx \|} - \frac{\By}{\|A \By\|}\|}  \sqrt{\|A\Bx\|^2 + \|A\By\|^2} \ge \sqrt{2}n e^t.
\end{equation}

Again by (i) and (ii) of Theorem \ref{theorem:deviation} we can assume that $\|A\Bx\|, \|A \By\| \le \sqrt{n} e^{t/4}$ with probability at least $1- e^{-ct^2n}$. Within this event

$$\frac{\|A(\frac{\Bx}{\|A \Bx \|} - \frac{\By}{\|A \By\|})\|}{\|\frac{\Bx}{\|A \Bx \|} - \frac{\By}{\|A \By\|}\|} \ge e^{t/2}\sqrt{n}.$$

\begin{claim}\label{claim:dim2:1}
Let $\Bx,\By \in S^{n-1}$ be given such that $\Bx \wedge \By \neq 0$. Then 

$$\P\Big(\exists \alpha, \beta, \alpha^2+\beta^2=1, \frac{\|A (\alpha \Bx + \beta \By)\|}{\| \alpha \Bx + \beta \By \|} \ge \sqrt{n} e^{t/2}\Big) \le e^{-ct^2n}.$$
\end{claim} 

It remains to verify Claim \ref{claim:dim2:1}. To this end, we first find a $n^{-C}$-net $\CM$ of the unit circle $S^1_{\Bx,\By}$ of the plane spanned by $\Bx$ and $\By$. As this set is one-dimensional, one can choose $|\CM| =n^{C}$. With $C$ chosen sufficiently large, one pass from the event 
$\{\exists \Bz\in S^1_{\Bx,\By}, \frac{\|A \Bz\|}{\| \Bz\|} \ge \sqrt{n} e^t\}$ to the event $\{\exists \Bz\in \CM, \frac{\|A \Bz\|}{\| \Bz\|} \ge \sqrt{n} e^t\}$ without any essential loss. However, for each fixed $\Bz$, by (i) of Theorem \ref{theorem:deviation} we have $\P(\frac{\|A \Bz\|}{\| \Bz\|} \ge \sqrt{n} e^{t/2}) \le e^{-ct^2n}$. The claim then just follows after taking union bound over $n^{O(1)}$ elements of $\CM$.

We complete the proof by proving (iii). This time, without loss of generality we assume 

$$0< \langle \Bx,\By \rangle 
\le 1.$$ 

We write

\begin{align*}
\frac{\|A\Bx \wedge A \By\|}{\|\Bx \wedge \By\|} = \|A \Bx\| \|A\By\| \frac{\|A\Bx/\|A\Bx\| \wedge A \By/\|A \By\|\|}{\|\Bx \wedge \By\|}.
\end{align*}

Thus our assumption implies that

\begin{align*}
\|A\Bx/\|A\Bx\| \wedge A \By/\|A \By\|\| \le \frac{n e^{-t}}{ \|A\Bx\| \|A\By\|}{\|\Bx \wedge \By\|}.
\end{align*}

By Fact \ref{fact:f}, we then have

\begin{equation}\label{eqn:dim2:upper:1'}
\|A(\Bx/\|A\Bx\| -\By/\|A \By\|)\| \le \frac{n e^{-t}}{ \|A\Bx\| \|A\By\|}{\|\Bx - \By\|}.
\end{equation}

Now use \eqref{eqn:dim2:upper:2}

\begin{align}\label{eqn:dim2:upper:2'}
&\frac{\|A(\frac{\Bx}{\|A \Bx \|} - \frac{\By}{\|A \By\|})\|}{\|\Bx-\By\|}=\frac{\|A(\frac{\Bx}{\|A \Bx \|} - \frac{\By}{\|A \By\|})\|}{\|\frac{\Bx}{\|A \Bx \|} - \frac{\By}{\|A \By\|}\|} \sqrt{1/\|A\Bx\|^2 + 1/\|A\By\|^2} \frac{\frac{1}{\sqrt{1/\|A\Bx\|^2 + 1/\|A\By\|^2}}\|{\frac{\Bx}{\|A \Bx \|} - \frac{\By}{\|A \By\|}}\|}{\|\Bx -\By\|}.
\end{align}

As $\langle \Bx, \By \rangle \ge 0$, for any $\alpha^2 +\beta^2=1$

$$\||\alpha| \Bx - |\beta| \By\|^2 = 1 - 2 |\alpha \beta| \langle \Bx,\By \rangle \ge  1 - \langle \Bx,\By \rangle = \frac{1}{2}\|\Bx -\By\|^2.$$

Thus

\begin{equation}\label{eqn:dim2:upper:3'}
\frac{\frac{1}{\sqrt{1/\|A\Bx\|^2 + 1/\|A\By\|^2}}\|{\frac{\Bx}{\|A \Bx \|} - \frac{\By}{\|A \By\|}}\|}{\|\Bx -\By\|} \ge \frac{1}{\sqrt{2}}.
\end{equation}

It follows from \eqref{eqn:dim2:upper:1'}, \eqref{eqn:dim2:upper:2'} and \eqref{eqn:dim2:upper:3'} that

\begin{equation}
\frac{\|A(\frac{\Bx}{\|A \Bx \|} - \frac{\By}{\|A \By\|})\|}{\|\frac{\Bx}{\|A \Bx \|} - \frac{\By}{\|A \By\|}\|}  \sqrt{\|A\Bx\|^2 + \|A\By\|^2} \le \sqrt{2}n e^{-t}.
\end{equation}

Again by (ii) and (iii) of Theorem \ref{theorem:deviation} we can assume $\|A\Bx\|, \|A \By\| \ge \sqrt{n} e^{-t/4}$ with probability at least $1-K^ne^{-tn/2}-e^{-c''t^2n}\}$. Within this event

$$\frac{\|A(\frac{\Bx}{\|A \Bx \|} - \frac{\By}{\|A \By\|})\|}{\|\frac{\Bx}{\|A \Bx \|} - \frac{\By}{\|A \By\|}\|} \le e^{-t/2}\sqrt{n}.$$

\begin{claim}\label{claim:dim2:1'}
Let $\Bx,\By \in S^{n-1}$ be given such that $\Bx \wedge \By \neq 0$. Then 

$$\P\Big(\exists \alpha, \beta, \alpha^2+\beta^2=1, \frac{\|A (\alpha \Bx + \beta \By)\|}{\| \alpha \Bx + \beta \By \|} \le \sqrt{n} e^{-t/2}\Big) \le e^{-ct^2n}.$$
\end{claim} 

The proof of Claim \ref{claim:dim2:1'} is similar to that of Claim \ref{claim:dim2:1}. In fact, consider the $n^{-C}$-net $\CM$ of the unit circle $S^1_{\Bx,\By}$ of the plane spanned by $\Bx$ and $\By$ with size $|\CM| =n^{C}$ and with sufficiently large $C$. As $t=o(\log n)$, one can pass the event $\{\exists \|\Bz\|=1, \frac{\|A \Bz\|}{\| \Bz\|} \le \sqrt{n} e^{-t/2}\}$ to  the event $\{\exists \Bz\in \CM, \frac{\|A \Bz\|}{\| \Bz\|} \le \sqrt{n} e^{-t/2}\}$ without any essential loss.
However, for each fixed $\Bz$, by (ii)  and (iii) of Theorem \ref{theorem:deviation} we have $\P(\frac{\|A \Bz\|}{\| \Bz\|} \le \sqrt{n} e^{-t/2}) \le \max\{K^n e^{-nt/2}, e^{-c''t^2n}\}$. The claim then just follows after taking union bound over $n^{O(1)}$ elements of $\CM$.

\end{proof}

\begin{remark} Although the behavior of $\frac{\|A \Bx \wedge A \By\|}{\|\Bx \wedge \By\|}$ is more relevant to our study, we had to pass to $\frac{\|A\Bx/\|A\Bx\| \wedge A \By/\|A\By\|\|}{\|\Bx \wedge \By\|}$ in both Theorem \ref{theorem:dim2:lower} and Theorem \ref{theorem:dim2:upper} to make use of the convenient identity \eqref{eqn:dim2:wedge}  (which is valid only for unit vectors.)

\end{remark}

\subsection{Step 3} Let $(\Bx_0,\By_0)$ be any vector pair from $\CP_{start}$. We will show

\begin{lemma}\label{lemma:top:dim2}
For any $t\ge 1/n$ we have 

$$\P\Big( |\frac{1}{N}  \log \|B_N \Bx_0 \wedge B_N \By_0\| | \ge t \Big) \le \exp(-c\min\{t^2, t\} Nn) + N n^{-cn}.$$
\end{lemma}

It is clear that Theorem \ref{theorem:main} follows from Lemma \ref{lemma:top:dim2} after taking union bound over $\CP_{start}$. 

To prove this result, we first give an analog of Theorem \ref{theorem:deviation}. Recall the notion of $\Bx_i, \By_i$ from Subsection \ref{subsection:(2)}. For short, denote
$$y_i:=\log \frac{\|A \Bx_i \wedge A\By_i \|}{\|\Bx_i \wedge \By_i\|}-\log n.$$

\begin{proof}(of Lemma \ref{lemma:top:dim2}) We will follow the proof of Lemma \ref{lemma:top}. First, by Theorem \ref{theorem:stable:2}, the event $\CG_1$ that $(\Bx_i,\By_i)\in \CP$ for all $1\le i\le N$ holds with probability 

$$\P(\CG_1) \ge 1 - N n^{-cn}.$$

Consider the random sum 

$$S =\frac{1}{N} (y_1+\dots+y_N).$$

Basing on Corollary \ref{cor:dim2:lower} and Theorem \ref{theorem:dim2:upper}, the event $\CG_2$ such that $|y_i|\le 2 \log D$ for all $y_i, 1\le i\le N$ satisfies 

$$\P(\CG_2) \ge 1-N D^{-n}.$$

Introduce the new random variables $y_i':= y_i 1_{|y_i| \le 2 \log D}$ and $y_i'' := {y_i'} - \E_{A_i} {y_i'}$. As usual, in the sequel we will be conditioning on $A_1,\dots, A_{i-1}$. By Theorem \ref{theorem:dim2:upper}, for any positive $t=O(1)$ 

\begin{equation}\label{eqn:y':1:dim2}
\P_{A_i}(|y_i'| \ge t) \le \P_{A_i}(|y_i|\ge t) \le e^{-ct^2n}.
\end{equation}

Also, by Theorem \ref{theorem:dim2:lower} and Theorem \ref{theorem:dim2:upper}, for $O(1) \le t\le 2 \log D$

\begin{equation}\label{eqn:y':2:dim2}
\P_{A_i}(|{y_i'}| \ge t) \le \P_{A_i}(|y_i|\ge t) \le C^n e^{-tn/2} + e^{-ct^2n}.
\end{equation}

Consequently,

$$\E_{A_i} |{y_i'}| \le \int_{0}^{2 \log D} t \P(|{y_i'}|\ge t) \le O(\int_{0}^{1/\sqrt{n}} t dt) =O(1/n).$$





Consider the martingale sum $S'':= \frac{1}{N}(y_1''+\dots+y_N'')$. By definition, $|y_i''|\le 2 \log D$. Also by \eqref{eqn:y':1:dim2} and \eqref{eqn:y':2:dim2}, for $t\ge 1/n$ 

$$\P_{A_i}(|y_i''|\ge t) \le \P_{A_i}(|{y_i'}|\ge t) \le  \exp(-c\min\{t^2, t\} n).$$ 

This implies that for $\lambda = ctn$, 

$$e^{-2\lambda t} \E (e^{\lambda y_i'' }|A_1,\dots,A_{i-1}), e^{-2\lambda t} \E (e^{-\lambda y_i'' }|A_1,\dots,A_{i-1}) \le \exp(-c \min \{t, t^2\}n).$$ 

From here, argue similarly as in Section \ref{section:step3}, for $t\ge 1/n$

$$\P(|S''|\ge 2t) = \P(|y_1''+\dots+y_N''| \ge 2Nt) \le \exp(-c'\min\{t^2, t\}Nn).$$

Thus, 


\begin{align*}
\P(|S| \ge 2t +O(1/n)) \le   \P(|S''| \ge2 t) + \P(\bar{\CG_1}  \cup \bar{\CG_2})  &\le \exp(-c\min\{t^2, t\}Nn) +\P(\bar{\CG_1}) + \P(\bar{\CG_2})\\
&\le \exp(-c\min\{t^2, t\}Nn) + N  n^{-cn}.
\end{align*}

\end{proof}

\section{The least  Lyapunov's exponent: proof of (3) of Theorem \ref{theorem:main}}\label{section:least}Recall from Subsection \ref{subsection:(3)} that 

\begin{align*}
\log \dist (\Bc_n, span(\Bc_i, i\neq n))&= \log \dist(B_N\Be_n, H_{B_N\Be_1,\dots,B_N\Be_{n-1}}) \\
& = \sum_{i=1}^N \log \dist(A_i \Bv_i, H_{A_i \dots A_1 \Be_1,\dots,A_i \dots A_1 \Be_{n-1}})\\
& = \sum_{i=1}^N \log  d_i,
\end{align*}

with

\begin{equation}\label{eqn:dist'} 
d_i^2:= \dist^2(A_i \Bv_i, H_{A_i \dots A_1 \Be_1,\dots,A_i \dots A_1 \Be_{n-1}}) = \frac{1}{\|A_i^{-1}\Bv_i\|_2^2} =\frac{1}{ \sum_j \sigma_{ij}^{-2} |\Bv_i^T \Bu_{ij}|^2 },
\end{equation}

where $\sigma_{ij}$ and $\Bu_{ij}$ are the singular values and (unit) singular vectors of the matrix $A_i$, and thus independent of $\Bv_i$. 

Our main goal is the following estimate on $\P(\CE_{\eps,1})$.

\begin{theorem}\label{theorem:least} For given $\eps>0$, there exists an absolute constant $C$ such that the following holds for sufficiently large $n$ and $N$
$$\P(\frac{1}{N}\sum_{i=1}^N \log d_i \le  -(1/2+\eps) \log n) =\exp(-N/2)C^n + N n^{-\omega(1)}.$$ 

\end{theorem}

We will prove Theorem \ref{theorem:least} by invoking a series of known results in RMT.

Firstly, we will use the following isotropic delocalization result from \cite[Theorem 2.16]{BEKYY}.

\begin{theorem}\label{theorem:isotropic} Let $\eps$ and $A>0$ be given numbers. Then the following holds for $n\ge n(\eps,A)$: for any fixed unit vector $\Bv$, with probability at least $1 - n^{-A}$

$$\sup_{1\le j\le n} |\Bv^T \Bu_{ij}| \le n^{-1/2+\eps}.$$

\end{theorem}

From now on we will condition on the event $\CE_{\eps,A}$ of Theorem \ref{theorem:isotropic}, in which case we have
$$ \sum_j \sigma_{ij}^{-2} |\Bv_i^T \Bu_{ij}|^2 \le n^{-1+2\eps} \sum_j \sigma_j^{-2}.$$

Secondly, by \cite[Claim 5.1]{NgV} the following holds with probability $1-n^{-\omega(1)}$,

$$\sum_{j=1}^{n- O(\log n)} \sigma_j^{-2} \ll \frac{n}{\log n}.$$

Thus  

$$\|A_i^{-1}\Bv_i\|^2 = \sum_j \sigma_{ij}^{-2} |\Bv_i^T \Bu_{ij}|^2 \le n^{-1+2\eps} (\frac{n}{\log n} +  \sigma_n^{-2} \log n) \le (1+n^{-1/2}\sigma_n^{-1})^2 n^{3\eps}.$$

Thirdly, we use the following bound from \cite{rv} and \cite{NgV}.
 
 \begin{lemma}\label{lemma:lower} As long as $\delta \ge \exp(-cn)$
 
 $$\P(\sigma_n \le \delta/n) \le C_0 \delta.$$
 \end{lemma}
 
Thus altogether we have
\begin{equation}\label{eqn:A_i:tail}
\P(\|A_i^{-1}\Bv_i\|^2 \ge \frac{n^{1+3\eps}}{\delta^2})  \ll \delta.
\end{equation}

Passing to distances, we obtain the following.

\begin{theorem}\label{theorem:Ai:lower} For any $\delta > \exp(-cn)$ we have
$$\P(d_i n^{1/2+2\eps} \le \delta ) \le C_0 \delta.$$
\end{theorem}

Now let 

$$X_i:= \log (d_i n^{1/2+2\eps}) \mathbf{1}_{d_i \ge \exp(-cn)}.$$ 

We have shown that for a given $t_0>0$ and for any $\delta>0$

$$\P\Big(\E(\exp(-t_0X_i)|A_1\dots, A_{i-1})  \ge (1/\delta)^{t_0} \Big) \ll \delta.$$

Hence there exists an absolute constant $C$ such that for any $0<t_0\le 1/2$ 

$$0\le \E(\exp(-t_0X_{i})|A_1,\dots, A_{i-1})\le C.$$

Next write 

\begin{align*}
 \P(X_1+\dots +X_N \le -Nt_0) &= \P(-X_1  -\dots -X_N \ge Nt_0)  \le \exp(-Nt_0) \E  \exp(-t_0(X_1+\dots +X_N)) \\
&\le \exp(-Nt_0) \E  \big([\exp(-t_0(X_1+\dots +X_{N-1})] \E (\exp(-t_0X_N)|A_1,\dots, A_{N-1})\big)\\
&\le C\exp(-Nt_0) \E  \big([\exp(-t_0(X_1+\dots +X_{N-1})]\big) 
\end{align*}

Repeat the machinery for $X_{N-1},\dots, X_1$, we thus obtain

$$\E  \exp(-t_0(X_1+\dots +X_N)) \le C^N.$$

In summary,

$$\P(\sum_{i=1}^N \log (d_i n^{1/2+2\eps}) \le -Nt_0) \le \exp(-Nt_0) C^n + N \exp(-cn).$$

Choosing $t_0=1/2$, we obtain Theorem \ref{theorem:least} after a proper scaling of $\eps$ (assuming $n,N$ sufficiently large).

\section{remarks}\label{section:remark}
We have considered product of $N$ iid random matrices where the $A_i$ can be singular with positive probability. Because of this, one has to assume $N$ not to be too large. The main bulk of the paper develops several ways to balance between the singularity and the generality  of the asymptotic Lyapunov' exponents. 

There are various  models (especially in connection to the study of Shr\"odinger operators in various lattices or to the study of random band matrices) where it is natural to study the large deviation type problem for general unimodular ensembles with either discrete or continuous  atom  distribution. One extremely convenient property of this model is that one does not have to worry about $N$ as the product matrices never vanish. On the other hand, the {\it mean-field} techniques used in our note do not seem to work. One simple candidate for future study is the symplectic model 

$$A_i = \begin{pmatrix} \lambda W_n -E  & -I_n \\ I_n & 0 \end{pmatrix}$$ 

with given parameter $E,\lambda$, where $W_n=(w_{ij})_{1\le i,j\le n}$ are random Wigner matrices of upper diagonal entries of variance $1/n$. 

It has been shown in \cite{GM}  that the Lyapunov exponents of this model (and for far more general models) are distinct. Furthermore, these exponents were estimated rather precisely by Sadel and Schulz-Baldes (see also \cite{Do}) as follows.

\begin{theorem}\cite[Proposition 8]{SSch} As long as $E =2 \cos \kappa \neq 0$ and $|E| <2$, then for $1\le d\le n$

$$\gamma_d = \lambda^2 \frac{1+ 2(n-d)}{8 \sin^2 \kappa} + O(\lambda^3).$$
\end{theorem}

It remains an interesting and challenging problem to obtain large deviation type estimates for this model.

{\bf Acknowledgement.} The author is grateful to prof. Thomas Spencer for invaluable comments and suggestions.

\end{document}